\documentclass[a4paper,10pt]{scrartcl}

\usepackage[english] {babel}
\usepackage{amsmath,amsthm}
\usepackage{amsfonts}
\usepackage{amssymb}
\usepackage{mathrsfs}
\usepackage{exscale}
\usepackage{mathrsfs}
\usepackage[authoryear]{natbib}
\usepackage{mathabx}
\usepackage{subfigure}
\usepackage{pgfplots}
\usepackage{booktabs}
\usepackage{tikz}
\usepackage{a4wide}
\newlength\fheight \newlength\fwidth

\newcommand{\R}{\mathbb{R}}
\newcommand{\N}{\mathbb{N}}

\renewcommand{\P}[2]{\mathbb{P}_{#1}\left[#2\right]}
\newcommand{\E}[2]{\mathbb{E}_{#1}\left[#2\right]}

\newcommand{\Cov}{\operatorname{Cov}}
\newcommand{\V}[1]{\mathbb{V}\left[#1\right]}
\newcommand{\oP}{o_{\mathbb P}}
\newcommand{\OP}{O_{\mathbb P}}

\newcommand{\toD}{\stackrel{\mathcal D}{\rightarrow}}
\newcommand{\p}{\mathrm{pen}}
\newcommand{\pen}[2]{\p_{#1}\left(#2\right)}

\newtheorem{theorem}{Theorem}[section]

\newtheorem{lem}[theorem]{Lemma}
\newtheorem{cor}[theorem]{Corollary}
\newtheorem{remark}[theorem]{Remark}
\newtheorem{example}[theorem]{Example}
\newtheorem{assumption}{Assumption}

\allowdisplaybreaks

\begin{document}

\usetikzlibrary{patterns}

\begin{center}
\begin{minipage}{.8\textwidth}
\centering 
\LARGE Multidimensional multiscale scanning in Exponential Families: Limit theory and statistical consequences\\[0.5cm]

\normalsize
\textsc{Claudia K\"onig}\\[0.1cm]
\verb+claudia.koenig@stud.uni-goettingen.de+\\
Institute for Mathematical Stochastics, University of G\"ottingen\\[0.1cm]

\textsc{Axel Munk}\\[0.1cm]
\verb+munk@math.uni-goettingen.de+\\
Institute for Mathematical Stochastics, University of G\"ottingen\\
and\\
Felix Bernstein Institute for Mathematical Statistics in the Bioscience, University of G\"ottingen\\
and\\
Max Planck Institute for Biophysical Chemistry, G\"ottingen, Germany\\[0.1cm]

\textsc{Frank Werner}\footnotemark[1]\\[0.1cm]
\verb+frank.werner@mpibpc.mpg.de+\\
Institute for Mathematical Stochastics, University of G\"ottingen\\
and\\
Max Planck Institute for Biophysical Chemistry, G\"ottingen, Germany
\end{minipage}
\end{center}

\footnotetext[1]{Corresponding author}

\begin{abstract}
{We} consider the problem of finding anomalies in a $d$-dimensional field of independent random variables $\{Y_i\}_{i \in \left\{1,...,n\right\}^d}$, each distributed according to a one-dimensional natural exponential family $\mathcal F = \left\{F_\theta\right\}_{\theta \in\Theta}$. Given some baseline parameter $\theta_0 \in\Theta$, the field is scanned using local likelihood ratio tests to detect from a (large) given system of regions $\mathcal{R}$ those regions $R \subset \left\{1,...,n\right\}^d$ with $\theta_i \neq \theta_0$ for some $i \in R$. We provide a unified methodology which controls the overall family wise error (FWER) to make a wrong detection at a given error rate.

Fundamental to our method is a Gaussian approximation of the distribution of the underlying multiscale test statistic with explicit rate of convergence. From this, we obtain a weak limit theorem which can be seen as a generalized weak invariance principle to non identically distributed data and is of independent interest. Furthermore, we give an asymptotic expansion of the procedures power, which yields minimax optimality in case of Gaussian observations.
\end{abstract}

\textit{Keywords:} exponential families, multiscale testing, invariance principle, scan statistic, weak limit, family wise error rate \\[0.1cm]

\textit{AMS classification numbers:} Primary 60F17, 62H10, Secondary 60G50, 62F03. \\[0.3cm]
\section{Introduction}
Suppose we observe an independent, $d$-dimensional field $Y$ of random variables
\begin{equation}\label{eq:model}
Y_i \sim F_{\theta_i}, \qquad i \in I_n^d := \left\{1,...,n\right\}^d,
\end{equation}
where each observation is drawn from the same given one-dimensional natural exponential family model $\mathcal F=\{F_\theta\}_{\theta \in \Theta}$, but with potentially different parameters $\theta_i$. Prominent examples include $Y_i$ with varying normal means $\mu_i$ or a Poisson field with varying intensities $\lambda_i$. Given some baseline parameter $\theta_0 \in \Theta$ (e.g. all $\mu_i=0$ for a Gaussian field), we consider the problem of finding anomalies (hot spots) in the field $Y$, i.e. we aim to identify those regions $R \subset I_n^d$ where $\theta_i \neq \theta_0$ for some $i \in R$. Here $R$ runs through a given family of candidate regions $R\in \mathcal R_n \subset \mathcal P(I_n^d)$ {where $\mathcal P(A)$ denotes the power set of a set $A$}. For simplicity, we will suppress the subindex $n$ whenever it is clear from the context, i.e. write $\mathcal R = \mathcal R_n$ in what follows. Such problems occur in numerous areas of application ranging from astronomy and biophysics to genetics engineering, specific examples include detection in radiographic images \citep{Kazantsev2002791}, genome screening \citep{jiang2016assessing} and object detection in astrophysical image analysis \citep{SourceDetection}, to mention a few. Our setting includes the important special cases of Gaussian {\citep{CastroDonoho1,sac16, kou2017identifying, cs17}}, Bernoulli \citep{Walther}, and Poisson random fields {\citep{Zhangetal, ketal05,rivera2013optimal,tu2013maximum}}. Extensions to models without exponential family structure as well as replacing the baseline parameter $\theta_0$ by a varying field of known baseline intensities can be treated as well (cf. Remark \ref{rem:model} below), but to keep the presentation simple, we restrict ourselves to the afore mentioned setting. 

\subsection{Methodology}

Inline with the above mentioned references (see also Section \ref{subsec:lit} {for a more comprehensive review}), the problem of finding hot spots is regarded as a multiple testing problem, i.e. many 'local' tests on the regions $\mathcal R$ are performed simultaneously, while keeping the overall error of wrong detections controllable. For a fixed region $R \in \mathcal R$ the likelihood ratio test (LRT) for the testing problem
\begin{equation}\label{eq:single_testing_hypothesis}\tag{$H_{R,n}$}
\forall~i \in R: \theta_i = \theta_0
\end{equation}
vs.
\begin{equation}\label{eq:single_testing_alternative}\tag{$K_{R,n}$}
\exists~i \in R \text{ s.t. }\theta_i \neq \theta_0,
\end{equation}
is a powerful test in general, and often known to have certain optimality properties (depending on the structure of $R$, see e.g. \citet{lr05}). Therefore, the LRT will always be considered throughout this paper as the 'local' test. We stress, however, that our methodology could also be used for other systems of local tests, provided they obey a sufficiently well behaving asymptotic expansion (see Remark \ref{rem:model}).
The LRT is based on the test statistic
\begin{equation}\label{eq:lr_stat}
T_R(Y,\theta_0):=\sqrt{2\log\left(\frac{\sup_{{\theta \in \Theta}}\prod_{i\in R}f_{\theta}(Y_{i})}{\prod_{i\in R}f_{\theta_0}(Y_{i})}\right)},
\end{equation}
where $f_\theta$ denotes the density of $F_\theta$, and $H_{R,n}$ is rejected when $T_R(Y,\theta_0)$ is too large. As it is not known a priori which regions $R$ might contain anomalies, i.e. for which $R \in \mathcal R$ the alternative \ref{eq:single_testing_alternative} might hold true, it is {required} to control the family wise error arising from the multiple test decisions of the local tests based on $T_R \left(Y, \theta_0\right)$, $R \in \mathcal R$. Obviously, without any further restriction on the complexity of $\mathcal R$ this error cannot be controlled. To this end, we will assume that the regions $R$ can be represented as a sequence of discretized regions in
\begin{equation}\label{eq:defi_R_n}
\mathcal{R}=\mathcal R_n := \left\{R \subset I_n^d ~\big|~ R = I_n^d \cap n R^* \text{ for some }R^* \in \mathcal R^*\right\}
\end{equation}
for some system of subsets (e.g.\ all hypercubes) of the unit cube $\mathcal R^* \subset \mathcal P(\left[0,1\right]^d)$, to be specified later. This gives rise to the sequence of multiple testing problems
\begin{equation}\label{eq:multiple_testing}
H_{R,n} \text{ vs. }K_{R,n} \qquad\mathbf{simultaneously}\text{ over } \mathcal R_n{, \quad n \in \mathbb{N}.}
\end{equation}
The aim of this paper is to provide methodology to control (asymptotically) the family wise error rate (FWER) $\alpha \in \left(0,1\right)$ when \eqref{eq:multiple_testing} is considered as a multiple testing problem, i.e. to provide a {sequence of} multiple test{s} $\Phi{:= \Phi_n}$ \citep[see e.g.][]{d14} for \eqref{eq:multiple_testing} such that
\begin{equation}\label{eq:FWER}
\sup\limits_{R \in \mathcal R_n} \P{H_{R,n}}{\Phi\text{ rejects any }H_{R',n}\text{ with }R' \subset R} \leq \alpha + o(1)
\end{equation}
as $n \to \infty$. In words, this ensures that the probability of making any wrong detection is controlled {at} level $\alpha$, as $n \rightarrow \infty.$

This task has been the focus of several papers during the last decades, for a detailed discussion see Section \ref{subsec:lit}.
We contribute to this field by providing a general theory for a unifying method in the model \eqref{eq:model} {including Gaussian, Poisson and Bernoulli observations.}
In view of \citep{accd11}, where observations from exponential families as in \eqref{eq:model} are {also} discussed, but the local tests are always as in the Gaussian case, we emphasize that our local tests are of type \eqref{eq:lr_stat}, hence exploiting the likelihood in the exponential family. This will result in improved power {and better finite sample accuracy} (see \citet{frick2014multiscale} for $d=1$). Our main technical contribution is to prove a weak limit theorem for the asymptotic distribution of our test statistic for general exponential family models as in \eqref{eq:model} and arbitrary dimension $d$. This can be viewed as a "multiscale" weak invariance principle for independent but not necessarily identically distributed r.v.'s. Further, we will provide an asymptotic expansion of the test's power which leads to minimax optimal detection of the test in specific models. 

{Throughout the following, we consider tests} of scanning-type, controlling the FWER by the maximum over the local LRT statistics in \eqref{eq:lr_stat}, i.e.
\begin{align}\label{MultiscaleStat}
T_n \equiv T_n (Y,\theta_0, \mathcal R_n,v):=\max_{R\in\mathcal R_n} \left[ T_R(Y,\theta_0)-\pen{v}{\left|R\right|}\right].
\end{align}
Here $\left|R\right|$ denotes the number of points in $R$. The values 
\begin{equation}\label{eq:defi_pen}
\pen{v}{r} := \sqrt{2 v \left(\log \left(n^d/ r\right) + 1\right)}
\end{equation}
where $\log$ denotes the natural logarithm, act as a scale penalization, {see also \citep{ds01, DuemWalt,Walther,frick2014multiscale}}. {This penalization with proper choice of $v$ guarantees} optimal detection power on all scales simultaneously as it prevents smaller regions from dominating the overall test statistic {(see Section \ref{subsec:power})}. {To obtain an a.s.\ bounded distributional limit for $T_n$ in \eqref{MultiscaleStat}, the} constant $v$ in \eqref{MultiscaleStat} can be any upper bound of the complexity of $\mathcal R^*$ measured in terms of the packing number (see {Assumption \ref{ass:v}} below). {For example, whenever} $\mathcal R^*$ has finite VC-dimension $\nu \left(\mathcal R^*\right)$, we can choose ${v} = \nu\left(\mathcal R^*\right)$. {However, we will see that} the test {has} better detection properties if $v$ is as small as possible (see Section \ref{subsec:power}). 
Hence, from this point of view it is advantageous to know exactly the complexity {of $\mathcal{R}^*$ in terms of the packing number}, a topic which has received less attention than computing VC-dimensions. Therefore, we compute the packing numbers for three important examples of $\mathcal R^*$, namely hyperrectangles, hypercubes and halfspaces explicitly in Appendix A.

\subsection{Overview over the results}

To construct a test which controls the FWER \eqref{eq:FWER}, we have to find a sequence of universal global thresholds $q_{1-\alpha,n}$ such that
\begin{align}\label{eq:prob_threshold}
\P{0}{T_n >q_{1-\alpha,n}} \leq \alpha + o(1),
\end{align}
where $\mathbb P_0 := \mathbb P_{H_{I_n^d,n}}$ corresponds to the case that no anomaly is present. Such a threshold suffices, as it can be readily seen from \eqref{MultiscaleStat} that
\begin{align*}
\sup\limits_{R \in \mathcal R_n} \P{H_{R,n}}{\Phi\text{ rejects any }H_{R',n}\text{ with }R' \subset R}  & \leq 
\sup\limits_{R \in \mathcal R_n} \P{H_{R,n}}{\Phi\text{ rejects }H_{R,n}} \\
& \leq \P{0}{\Phi\text{ rejects }H_{I_n^d,n}}.
\end{align*}
{Given $q_{1-\alpha,n}$,} the multiple test will reject whenever $T_n \geq q_{1-\alpha,n}$, and each local test rejects if $T_R \left(Y, \theta_0\right) \geq q_{1-\alpha,n} + \pen{v}{\left|R\right|}$. Due to \eqref{eq:FWER} and \eqref{eq:prob_threshold}, this will not be the case with (asymptotic) probability {$\leq \alpha$} for any $R \in \mathcal R_n$ such that $H_{R,n}$ holds true. 

To obtain the thresholds $q_{1-\alpha,n}$ we provide a Gaussian approximation of the scan statistic \eqref{MultiscaleStat} under $\mathbb P_0$ given by
\begin{equation}\label{eq:Gaussian_approximation}
M_n \equiv M_n\left(\mathcal R_n, v\right) := \max_{R\in\mathcal R_n} \left[ \left|R\right|^{-1/2} \left| \sum\limits_{i \in R} X_i \right| - \pen{v}{\left|R\right|} \right]
\end{equation}
with i.i.d. standard normal r.v.'s $X_i$, $i \in I_n^d$. We also give a rate of convergence of this approximation (Thm. \ref{thm:approximation}), which is determined by the smallest scale in $\mathcal R_n$. Based on these results, we obtain the $\mathbb P_0$-limiting distribution of $T_n$ as that of
\begin{equation}\label{eq:Gaussian_limit}
M\equiv M\left(\mathcal{R}^*, v\right):= \sup\limits_{R^* \in \mathcal{R}^*} \left[ \frac{|W(R^*)|}{\sqrt{|R^*|}} - \pen{v}{n^{d} \left|R^*\right|}\right] < \infty \text{ a.s.},
\end{equation}
where $W$ is white noise on $\left[0,1\right]^d$ and (with a slight abuse of notation) $\left|R^*\right|$ denotes the Lebesgue measure of $R^* \in \mathcal R^*$. This holds true as soon as $\mathcal R^*$ and $\mathcal R_n$ have a finite complexity, $\mathcal R^*$ consists of sets with a sufficiently regular boundary (see Assumption \ref{ass:R_complex}(b) below), and the smallest scales $|R_n|$ of the system $\mathcal{R}_n$ are restricted suitably, see \eqref{eq:r_n} below and the discussion there. 

In case of $\mathcal R^*$ being the subset of all hypercubes, we will also give an asymptotic expansion of the above test's power, which allows to determine the necessary average strength of an anomaly such that it will be detected with asymptotic probability $1$. This is only possible due to the penalization in \eqref{MultiscaleStat}, as otherwise the asymptotic distribution is not a.s. finite. If the anomaly is sufficiently small, we show that the anomalies which can be detected with asymptotic power one by the described multiscale testing procedure are the same as those of the oracle single scale test, which knows the size (scale) of the anomaly in advance. This generalizes findings of \citet{sac16} to situations where not only the mean of the signal is allowed to change, but its whole distribution. Furthermore, if the observations are Gaussian, {and $\mathcal R^*$ is the system of squares, our test with the proper choice $v = 1$ (see Example \ref{ex:sets})} achieves the asymptotic optimal detection boundary, i.e. no test can have larger power in a minimax sense, asymptotically. 

\subsection{Computation}

Note that {the} weak {limit $M$ of $T_n$} in \eqref{eq:Gaussian_limit} does not depend on any unknown quantity, and hence can be e.g. simulated generically in advance for any given system $\mathcal R$ as soon as a bound for the complexity of $\mathcal R^*$ can be determined. {If the system $\mathcal R$ has special convolution-type structure, we discuss an efficient implementation using fast Fourier transforms in Section \ref{sec:implementation} with computational complexity $\mathcal O \left(d \#\text{ scales } n^d \log n\right)$ for a single evaluation of $T_n$ or $M_n$. Once the quantiles are pre-computed, this allows for fast processing of incoming data sets.}

\subsection{Literature review and connections to existing work}\label{subsec:lit}~
Scan statistics and scanning-type procedures based on the maximum over an ensemble of local tests
have received much attention in the literature over the past decades. To determine the quantile, a common option is to approximate the tails of the asymptotic distribution suitably, as done e.g. by \citet{sv95,sy00, nw04,pgks05,fang2016poisson} for $d = 1$, by \citet{hp06} for $d=2$, and by \citet{j02} in arbitrary dimensions.
If the random field is sufficiently smooth (in contrast to the setting here) the Gaussian kinematic formula or similar tools can be employed, see e.g. \citet{adler2000excursion},  \citet{taylor2007detecting}, \citet{sga11}, \citet{cs17}. 
We also mention \citet{Alm1998}, who considers the situation of a fixed rectangular scanning set in two and three dimensions. In all these papers, no penalization has been used, which automatically leads to a preference for small scales of order $\log(n)$ \citep[see e.g.][]{KabluchkoMunk:2009} and to an extreme value limit, in contrast to {the weak invariance principle type limit} \eqref{eq:Gaussian_limit}. \citet{arias2017distribution} study the case of an unknown null distribution and propose a permutation based approximation, which is shown to perform well in the natural exponential family setting \eqref{eq:model}, however, only for $d = 1$. {Technically, mostly} related to our work are weak limit theorems for scale penalized scan statistics, which have e.g. been obtained by \citet{frick2014multiscale} and \citet{sac16}. However, these results are either limited to special situations such as Gaussian observations, or to $d = 1$. If a limit exists, the quantiles of the finite sample statistic can be {used to bound} the quantiles of the limiting ones as e.g. done by \citet{ds01, rivera2013optimal,ds18}.

Our results can be interpreted in both ways as we provide a Gaussian approximation of the scan statistic in \eqref{MultiscaleStat} by \eqref{eq:Gaussian_approximation} \textit{and} that we obtain \eqref{eq:Gaussian_limit} as a weak limit. 

Weak limit{s} for $T_n$ as in {\eqref{eq:Gaussian_limit}} are immediately connected to those for partial sum processes. Classical KMT-like approximations \citep[see e.g.][]{komlos1976approximation,Rio,Massart:1989} provide in fact a strong coupling of the whole process $\left(T_R\left(Y, \theta_0\right)\right)_{R \in \mathcal R_n}$ to a Gaussian version. Results of this form have been employed for $d=1$ previously in \citet{SHMunkDuem2013,frick2014multiscale}. Proceeding like this for general $d$ will restrict the system $\mathcal{R}_n$ to scales $r_n$ s.t. $\left|R\right|  \geq r_n$ where
\begin{align}\label{eq:KMT_r_n}
n^{d-1} \log(n) = o\left(r_n\right)
\end{align}
as $n \to \infty$, which is unfeasibly large for $d\geq 2$. Therefore, we take a different route and employ a coupling of the maxima in \eqref{MultiscaleStat} and \eqref{eq:Gaussian_approximation}, which relies on recent results by \citet{chernozhukov2014gaussian}, see also \citet{proksch2016multiscale}. However, in contrast to the present paper, they do not consider the local LR statistic and require that {$\left|R\right| = o\left(n^d\right)$ for all $R$.} This {excludes large scales and} leads to an extreme value type limit in contrast to $\eqref{eq:Gaussian_limit}$ which incorporates all (larger) scales. To make use of Chernozhukov et al.'s (2014) coupling results in our general setting, we provide a symmetrization-like upper bound for the expectation of the maximum of a partial sum process by a corresponding Gaussian version, cf. {Lemma} \ref{prop:partialsum}. Doing so we are able to approximate the distribution of $T_n$ in \eqref{MultiscaleStat} by \eqref{eq:Gaussian_approximation} as soon as we restrict ourselves to $R \in \mathcal R_n$ with $\left|R\right| \geq r_n$ where the smallest scales only need to satisfy {the lower scale bound (LSB)
	\begin{equation}\label{eq:r_n}
	\log^{12}(n) = o\left(r_n\right) \text{ as }n \to \infty,
	\end{equation}}
which compared to \eqref{eq:KMT_r_n} allows for considerably smaller scales whenever $d \geq 2$. Note that \eqref{eq:r_n} does not to depend on $d$. However, as we consider {the discretized sets in $I_n^d$ here, the corresponding lower bound $a_n$ for sets in $\mathcal R^* \subset \mathcal P (\left[0,1\right]^d)$} is $n^{-d}\log^{12}(n) = o \left(a_n\right)$, which in fact depends on $d$ as now the volume of the largest possible set has been standardized to one (see \eqref{eq:defi_R_n} and Theorem \ref{thm:asymptPower} below) and coincides with the sampling rate $n^{-d}$ up to a poly-$\log$-factor. In contrast, \eqref{eq:KMT_r_n} gives $n^{-1}\log(n) = o\left(a_n\right)$, independent of $d$, which only for $d=1$ achieves the sampling rate $n^{-d}$. Under \eqref{eq:r_n} we also obtain $\OP \left(\left(\log^{12}(n)/r_n\right)^{1/10}\right)$ as rate of convergence of this approximation (see \eqref{eq:approx_in_prob} below). 

Also the asymptotic power of scanning-type procedures has been discussed in the literature. An early reference is \cite{CastroDonoho1}, who provide a test for $d = 1$ achieving optimal detection power on the smallest scale. However, to obtain optimal power on all scales, a scale dependent penalization is necessary. We mention \citet{Walther}, who achieves this for the detection of spatial clusters in a two dimensional Bernoulli field by scale adaptive thresholding of local test statistics. \citet{bi13} for $d=2$ and \citet{kou2017identifying} for general $d$ provide optimality of scanning procedures for Gaussian fields. Based on \citet{kabluchko2011extremes}, \citet{sac16} provide asymptotic power expansions for the multiscale statistic in \eqref{MultiscaleStat} with a slightly different penalization, yielding minimax optimality in case of $d$-dimensional Gaussian fields. Inspired by their, however incomplete, proof, we are able to generalize these results in case of $\mathcal R^*$ being the set of all hypercubes to {the} exponential family model, \eqref{eq:model}, despite the fact that under the alternative the whole distribution in \eqref{eq:model} might change, whereas for Gaussian fields typically only the mean changes. Doing so we obtain sharp detection {boundaries}, which are known to be minimax in the Gaussian situation, if the parameter $v$ in the penalization \eqref{eq:defi_pen} is chosen to be equal to the packing number of the system of hypercubes. In contrast, if $v$ is chosen to be the VC-dimension, the detection power turns out to be suboptimal. This emphasizes the importance of knowledge of the packing number explicitly, {for an illustration cf. Example \ref{ex:cub_rect}}.

{Finally we also mention weaker error measures such as the false discovery rate (FDR) as a potential alternative to FWER control and hence more sensitive tests are to be expected} \citep[see e.g.][]{bh95,by01,lms16}. {However, this }is a different task and beyond the scope of our paper.

\section{Theory}
In this section we will summarize our theoretical findings. In Section \ref{subsec:setting} we give an overview and details on our precise setting and present our assumptions on the set of candidate regions $\mathcal R^*$. Section \ref{subsec:limit} provides the validity of the Gaussian approximation in \eqref{eq:Gaussian_approximation} and determines the $\mathbb P_0$-limiting distribution of $T_n$. In Section \ref{subsec:power} we derive an asymptotic expansion of the detection power. {Throughout this paper, the constants appearing might depend on $d$.}

\subsection{Setting and Assumptions}\label{subsec:setting}
In the following we assume that $\mathcal{F}=\{F_\theta\}_{\theta \in \Theta}$ in \eqref{eq:model} is a one-dimensional exponential family, which is regular and minimal, i.e. the {$\nu$-}densities of $F_\theta$ are of the form $f_\theta(x) = \exp \left({\theta x} - \psi(\theta)\right)$, the natural parameter space
\[
\mathcal{N}= \left\lbrace \theta \in \R^d: \int_{\R^d} \exp(\theta x) \,{\mathrm d\nu(x)} <\infty \right\rbrace
\]
is open and the cumulant transform $\psi$ is strictly convex on $\mathcal{N}$. Then, the moment generating function exists and the random variables $Y_i$ have sub-exponential tails, see \citet{casella2002statistical} and \citet{brown1986fundamentals} for details. {We further assume that $\operatorname{Var}Y_i >0.$}

\begin{example}\label{ex:models}
Let us discuss three important examples of the model \eqref{eq:model}.
\begin{enumerate}
\item Gaussian fields: Let $Y_i \sim \mathcal N\left(\theta, \sigma^2\right)$ where the variance $\sigma^2>0$ is fixed.  In this case, $\psi(\theta)= \frac{1}{2} \theta^2$, and 
\[
T_R \left(Y, \theta_0 \right)= 
\sqrt{|R|}\frac{ \left|\overline{Y}_R - \theta_0 \right|}{\sigma}
\]
\item Bernoulli fields: Let $Y_i \sim \text{Bin}\left(1,p\right)$ with $p \in \left(0,1\right)$. {Note, that w.l.o.g.\ the} cases $p=0$ and $p=1$ {are} excluded {as} in these cases one would screen the field correctly, anyway. The natural parameter is $\theta = \log\left(p/(1-p)\right)$, and using $\psi \left(\theta\right)= \log\left(1+ \exp\left(\theta\right)\right)$ we compute \begin{align*}
T_R \left(Y, \theta_0\right)= \sqrt{2 |R| \left[\overline{Y}_R \log\left( \frac{\overline{Y}_R}{\frac{\exp(\theta_0)}{1+ \exp(\theta_0)}}\right) {+} (1-\overline{Y}_R) \log\left(\frac{1- \overline{Y}_R}{\frac{1}{\exp(\theta_0)+1}} \right) \right]}. 
\end{align*}
\item Poisson fields: Let $Y_i \sim \text{Poi}(\lambda)$ with $\lambda \in \mathbb R$. Again, $\lambda = 0$ has to be excluded, but this case is again trivial. The natural parameter is $\theta = \log\left(\lambda\right)$, and using $\psi\left(\theta\right) = \exp\left(\theta\right)$ we compute 
\begin{align*}
T_R(Y, \theta_0)= \sqrt{2 |R| \left[\overline{Y}_R \log\left(\frac{\overline{Y}_R}{\exp( \theta_0)}\right) - (\overline{Y}_R- \exp(\theta_0)) \right]}.
\end{align*}
\end{enumerate}
\end{example}

To derive the Gaussian approximation \eqref{eq:Gaussian_approximation} of $T_n$ in \eqref{MultiscaleStat}, we need to restrict the cardinality of $\mathcal R_n$:
\begin{assumption}[Cardinality of $\mathcal R_n$]\label{ass:R_card}
There exist constants $c_1, c_2>0$ such that
\begin{equation}\label{eq:finite_cardinality}
\#\left(\mathcal R_n \right) \leq c_1 n^{c_2}.
\end{equation}
\end{assumption}

To furthermore control the supremum in \eqref{eq:Gaussian_limit}, we have to restrict the system of regions $\mathcal R^*$ suitably. To this end, we introduce some notation. 

For a set $R^* \in \mathcal R^*$ and $x \in \left[0,1\right]^d$ we define $d \left(x,\partial R^*\right) :=\inf_{y \in \partial R^*} \left\Vert x-y\right\Vert_2$ where $\partial R^*$ denotes the topological boundary of $R^*$, i.e. $\partial R^* = \overline{R^*} \setminus \left(R^*\right)^\circ$. Furthermore we define the $\epsilon-$annulus $R^*(\epsilon)$ around the boundary of $R^*$ for some $\epsilon >0$ as
\[
R^*(\epsilon) := \left\{ x \in \left[0,1\right]^d ~\big|~ d\left(x,\partial R^*\right) < \epsilon\right\}.
\]
{In the following we will consider the symmetric difference 
\[
R_1^* \bigtriangleup R_2^* := \left(R_1^* \cup R_2^*\right) \setminus \left(R_1^* \cap R_2^*\right), \qquad R_1^*, R_2^* \in \mathcal{R}^*
\]
and the corresponding metric 
\begin{equation}\label{eq:metric}
\rho^*\left(R_1^*, R_2^*\right) := \sqrt{\left|R_1^* \bigtriangleup R_2^*  \right|}, \qquad \text{for}\qquad R_1^*, R_2^* \in \mathcal R^*.
\end{equation}

To derive the weak limit of $T_n$, we need to restrict the system $\mathcal R^*$ further. Recall the VC-dimension \citep[see e.g.][]{VaartWellner:1996}.
\begin{assumption}[Complexity and regularity of $\mathcal R^*$]\label{ass:R_complex}~
\begin{itemize}
\item[(a)] The VC-Dimension $\nu \left(\mathcal R^*\right)$ of the set $\mathcal R^*$ is finite. 
\item[(b)] There exists some constant $C>0$ such that $|R^*(\epsilon)| \leq C \epsilon$ for all $\epsilon>0$ and all $R^* \in \mathcal R^*$ with the Lebesgue measure $\left|\cdot\right|$.
\end{itemize}
\end{assumption}
Finally, to ensure a.s. boundedness of the limit in \eqref{eq:Gaussian_limit}, we will furthermore require that $v$ in \eqref{MultiscaleStat} is chosen appropriately{. To this end we introduce the packing number $\mathcal{K}(\epsilon, \rho, \mathcal W)$ of a subset $\mathcal W$ of $\mathcal{R}^*$ w.r.t. a metric $\rho$, which is given by the maximum number $m$ of elements $W_1, \ldots, W_m \in \mathcal{W}$ s.t. $\rho(W_i, W_j)> \epsilon$ for all $i \neq j$,  
i.e. by the largest number of $\epsilon$-balls w.r.t. $\rho$ which can be packed inside $\mathcal W$, see e.g. \citet[Def. 2.2.3]{VaartWellner:1996}.}
\begin{assumption}[Choice of $v$]\label{ass:v}
The constant $v$ in \eqref{MultiscaleStat} and \eqref{eq:defi_pen} is chosen such that there exist constants $k_1, k_2>0$ such that
\begin{equation}\label{eq:capacity_bound}
\mathcal{K}\left((\delta u)^{1/2}, \rho^*, \{R \in \mathcal{R}^*: |R|\leq \delta\} \right) \leq k_1 u^{-k_2} \delta^{-v}
\end{equation}
for all $u, \delta \in (0,1]$ with $\rho^*$ as in \eqref{eq:metric}.
\end{assumption}

Let us briefly comment on the above assumptions. 
\begin{remark}\label{rem:ass}~
\begin{itemize}
\item Assumption \ref{ass:R_card} will allow us to apply recent results by \citet{chernozhukov2014gaussian} to couple the process in \eqref{MultiscaleStat} with a Gaussian version as in \eqref{eq:Gaussian_approximation}. Note that Assumption \ref{ass:R_complex}(a) immediately implies Assumption \ref{ass:R_card}. 
\item We stress that the Assumption \ref{ass:R_complex}(b) is satisfied whenever $\mathcal{R}^*$ consists of regular Borel sets $R^*$ only, i.e. each $R^* \in \mathcal R^*$ is a Borel set and $\left|\partial R^*\right| = 0$ for all $R^* \in \mathcal R^*$.
\item Note that Assumption \ref{ass:R_complex}(a) also implies that $v = \nu\left(\mathcal R^*\right)$ is a valid choice in the sense of Assumption \ref{ass:v}. This basically follows from the relationship between capacity and covering numbers and {a bound on covering numbers from} \citet[Thm. 2.6.4]{VaartWellner:1996}. However, \eqref{eq:capacity_bound} might also be satisfied for considerably smaller numbers $v$ (see the examples below).
\end{itemize}
\end{remark}

}

\begin{example}\label{ex:sets}~
	\begin{enumerate}
		\item Consider the set $\mathcal S^*$ of all hyperrectangles in $\left[0,1\right]^d$, i.e. each $S^* \in \mathcal S^*$ is of the form $S^* = \left[s,t\right] := \left\{x \in \left[0,1\right]^d ~\big|~ s_i \leq x_i \leq t_i \text{ for } 1 \leq i \leq d\right\}$. {Obviously, the corresponding discretization $\mathcal S_n$ consists of hyperrectangles in $I_n^d$, which are determined by their upper left and lower right corners, i.e. $\#\left(\mathcal S_n \right) \leq n^{2d}$, which proves Assumption \ref{ass:R_card}. According to \citet[][Ex. 2.6.1]{VaartWellner:1996} we have $\nu\left(\mathcal S^*\right) = 2d$, and as $\mathcal S^*$ consists only of regular Borel sets, also Assumption \ref{ass:R_complex} is satisfied. In Appendix A we give a simple argument that Assumption \ref{ass:v} holds true whenever $v>  2d-1$. Employing more refined computations, it can even be shown that $v = 1$ is a valid choice if we allow for additional powers of $\left(-\log\left(\delta\right)\right)$ on the right-hand side of \eqref{eq:capacity_bound}, see Theorem 1 in \citet{Walther} or Lemma 2.1 in \citet{ds18}.}
		\item We may also consider the (smaller) set $\mathcal Q^*$ of all hypercubes in $\left[0,1\right]^d$, i.e. each $Q^* \in \mathcal Q^*$ is of the form $\left[t,t+h\right]$ with $t \in \left[0,1\right]^d$ and $0 < h \leq 1- \max_{1 \leq i \leq d} t_i$. As $\mathcal Q^* \subset \mathcal S^*$, {Assumptions \ref{ass:R_card} and \ref{ass:R_complex} are satisfied. Refined computations in Appendix A show that $v=1$ is a valid choice in the sense of Assumption \ref{ass:v}, independent of $d$ (opposed to the VC-dimension $\nu\left(\mathcal Q^*\right) = \lfloor \frac{3d+1}{2} \rfloor$ according to \citet{despres2014vapnik}).}
		\item Let $\mathcal{H}^*$ be the set of all half-spaces in $\left[0,1\right]^d$, i.e.
		\begin{align*}
		\mathcal{H}^* := \left\lbrace H_{a, \alpha} \left| \right. \alpha \in \R, a \in \mathbb{S}^{d-1}\right\rbrace,\quad H_{a, \alpha} := \left\lbrace x \in \left[0,1\right]^d ~\big|~ \langle x, a \rangle \geq \alpha \right\rbrace.
		\end{align*}
		The VC-dimension of $\mathcal H^*$ is $\leq d+1$ \citep[see e.g.][Cor. 4.2]{devroye2012combinatorial}, which proves that {Assumptions \ref{ass:R_card} and \ref{ass:R_complex} are satisfied. On the other hand, we prove in Appendix A that $v= 2$ satisfies Assumption \ref{ass:v}.}
	\end{enumerate}
\end{example}

\begin{remark}
	As discussed in the introduction, we will show in case of hypercubes that a smaller value of $v$ in Assumption \ref{ass:v} and hence in \eqref{eq:defi_pen} will lead to a better detection power. More precisely, only for $v =1$ we will obtain minimax optimality in a certain sense (see Section \ref{subsec:power} below). In case of hyperrectangles this is more involved, but it can, however, be argued along \citet{Walther} that for $d = 2$ the choice $v = 1$ yields minimax optimality also in this situation for specific sequences of rectangles.
\end{remark}

\subsection{Limit theory}\label{subsec:limit}

Now we are in position to show that the quantiles of the multiscale statistic in \eqref{MultiscaleStat} can be approximated uniformly by those of the Gaussian version in \eqref{eq:Gaussian_approximation}, and furthermore that $M_n\left(\mathcal R_n, v\right)$ in \eqref{eq:Gaussian_approximation} converges to a non-degenerate limit {whenever $v$ satisfies Assumption \ref{ass:v}}. For the former we require a lower bound on the smallest scale as given in \eqref{eq:r_n}. Given a discretized set of candidate regions $\mathcal R_n \subset \mathcal P\left(I_n^d\right)$ and $c>0$ we introduce 
\[
\mathcal R_{n| c} := \left\{R \in \mathcal R_n ~\big|~ \left|R\right| \geq c\right\}.
\]
With this notation we can formulate our main theorem{s}.
\begin{theorem}[Gaussian approximation]\label{thm:approximation}
Let $Y_i$, $i \in I_n^d$ be a field of random variables as in \eqref{eq:model}, let $\mathcal{R}^*$ be a set of candidate regions {satisfying Assumption \ref{ass:R_card}} and let $(r_n)_n\subset(0,\infty)$ be a sequence such that the LSB \eqref{eq:r_n} holds true. Let $v \in \mathbb R$ be fixed. 
\begin{itemize}
\item[(a)] {Then under $\mathbb{P}_0$}
\begin{equation}\label{eq:approx_in_prob}
T_n\left(Y, \theta_0, \mathcal R_{n | r_n}, v\right) - M_n\left(\mathcal R_{n | r_n}, v\right)  = \OP\left( \left(\frac{\log^{12}(n)}{r_n} \right)^{1/10}\right)
\end{equation}
as $n \to \infty$ with $M_n$ as in \eqref{eq:Gaussian_approximation}.
\item[(b)] For all $q \in \R$ we have
\begin{equation}\label{eq:quantile_approx}
\lim\limits_{n \to \infty} \left|\P{0}{T_n\left(Y, \theta_0, \mathcal R_{n | r_n}, v\right) > q} - \P{}{M_n\left(\mathcal R_{n | r_n}, v\right) > q}\right| = 0.
\end{equation}
\end{itemize}
\end{theorem}
{Note that $M_n$ does not depend on any unknown quantities and can e.g. be simulated for fixed $n$, see Section \ref{sec:numerics} for details. Beyond this, we can now also derive a weak limit of $T_n$. }
\begin{theorem}[Weak $\mathbb P_0$ limit]\label{thm:weak_limit}
{Under the Assumptions of Theorem \ref{thm:approximation} suppose that also Assumption \ref{ass:R_complex} is satisfied. Then it holds for any fixed $v \in \mathbb R$} under $\mathbb P_0$ that
\begin{equation}\label{eq:weak_limit}
T_n\left(Y, \theta_0, \mathcal R_{n | r_n}, v\right)\toD  M \left(\mathcal R^*, v\right) \qquad\text{as}\qquad n \to \infty,
\end{equation}
with $M \left(\mathcal R^*, v\right)$ as in \eqref{eq:Gaussian_limit}. {If $v$ furthermore satisfies Assumption \ref{ass:v},} then $M(\mathcal{R}^*, v)$ is almost surely finite and non-degenerate.
\end{theorem}
{Note that our proof of Theorem \ref{thm:weak_limit} explicitly requires the VC-dimension $\nu\left(\mathcal R^*\right)$ to be finite, and it is not clear if this Assumption could be dropped.}

{\begin{example}[Gaussian approximation in the hyperrectangle / hypercube case]\label{ex:cub_rect}
Recall Example \ref{ex:sets} and let $\mathcal S^*$ be the set of all hyperrectangles and $\mathcal Q^*$ be the set of all hypercubes in $\left[0,1\right]^d$. 
Then for any sequence $r_n$ satisfying the LSB \eqref{eq:r_n} the approximation \eqref{eq:approx_in_prob} holds under $\mathbb P_0$ for $\mathcal S_{n | r_n}$ and $\mathcal Q_{n | r_n}$, respectively. Monte-Carlo simulations (by means of \eqref{eq:Gaussian_approximation} with $n = 128$ and $d = 2$) of the densities of $M_n$ with different values of $v$ are shown in Figure \ref{fig:distributions}. The smallest possible values of $v$ which we may choose according to Example \ref{ex:sets} are $v = 3 +\epsilon$ and $v = 1$, respectively. The corresponding results are depicted in the first picture of of Figure \ref{fig:distributions} with $\epsilon = 0$ for simplicity. Alternatively, we may use the VC-dimensions $\nu\left(\mathcal S^*\right) = 4$ and $\nu\left(\mathcal Q^*\right) = 3$ respectively, which lead to the simulated {densities of $M_n$} shown in the bottom row of Figure \ref{fig:distributions}. Note that the distributions of $M_n \left(\mathcal S^n,4\right)$ and $M_n \left(\mathcal Q_n,3\right)$ are extremely close, which somewhat contradicts the intuition that detection in the less complex system of squares should be notably easier than detection in the system of all rectangles. The explanation for this is that $v = 3$ clearly overpenalizes the system $\mathcal Q_n$ of squares. In contrast, if the penalization is chosen according to the smallest possible values satisfying Assumption \ref{ass:v} (which allows for minimax detection in the system of squares, cf. Corollary \ref{cor:power} below), then the densities differ substantially.

\begin{figure}[!htb]
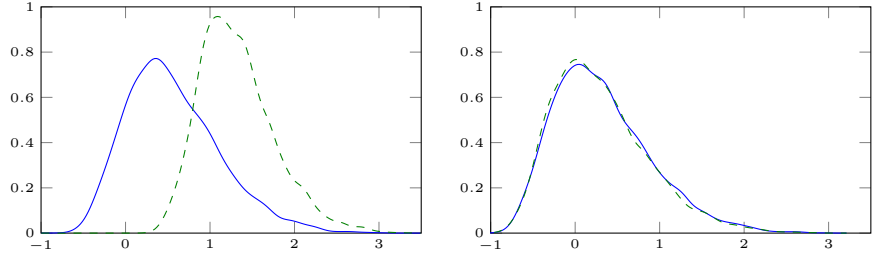

\setlength\fheight{3cm} \setlength\fwidth{5cm}
\centering
\tiny
\begin{tabular}{cc}
\input{comparison.tikz} & \input{comparison_VC.tikz}
\end{tabular}
\caption{Simulated densities of the Gaussian approximations, displayed by a standard kernel estimator obtained from $10^4$ runs of the test statistic \eqref{eq:Gaussian_approximation} ($M_n \left(\mathcal S^n,v\right)$ (\ref{rectangles}), $M_n \left(\mathcal Q_n,v\right)$ (\ref{squares})).
Left: optimal calibration with the covering number $v = 3$ and $v = 1$, respectively. 
Right: alternative calibration using the VC-dimension $\nu\left(\mathcal S^*\right) = 4$ and $\nu\left(\mathcal Q^*\right) = 3$.}
\label{fig:distributions}
\end{figure}
\end{example}}

{
\begin{remark}(Beyond exponential families)\label{rem:model}
\begin{enumerate}
\item[a)] Obviously, $\theta_0 \in \Theta$ can be replaced by a field $\left(\theta_i\right)_{i \in I_n^d}$ of known baseline parameters.
\item[b)] The proofs of Theorem \ref{thm:approximation} and Theorem \ref{thm:weak_limit} rely on a third-order Taylor expansion of $T_R$ and on the sub-exponential tails of the random variables $Y_i$, but not explicitly on the exponential family structure. Therefore, if in more general models corresponding assumptions are posed \citep[see also][Sec. 2.2]{arias2017distribution}, our results do immediately generalize to the case that the observations are not drawn from an exponential family as in \eqref{eq:model}. As an example, suppose our observations are drawn from the Weibull distribution with fixed scale parameter $\lambda>0$ and variable shape parameter $\theta>0$, i.e.
\begin{equation}\label{eq:weibull}
f_\theta(x)= \left(\frac{\theta}{\lambda} \right) \left(\frac{x}{\lambda}\right)^{\theta-1} \exp \left( -\left(\frac{x}{\lambda} \right)^\theta\right), \qquad x\geq0. 
\end{equation}
It is well-known that $\left\{f_\theta\right\}_{\theta >0}$ is not an exponential family. However, it is clear from \eqref{eq:weibull} that the likelihood-ratio test statistics $T_R$ are arbitrary smooth, i.e. a third order Taylor expansion is valid. If we restrict to $\theta \geq 1$ (non-decreasing failure rate), we immediately obtain sub-exponential tails, the MLE is unique and for $\theta \geq 2$ one also has asymptotic normality \citep[see e.g.][]{s85,fb97}. As a consequence, a similar coupling result as in Theorem \ref{lemma:Chernozhukov_NEF} below is possible, which would yield analogs to Theorems \ref{thm:approximation} and \ref{thm:weak_limit} also in this non exponential family situation. We emphasize that also Theorem \ref{thm:asymptPower} below can be generalized accordingly.
\end{enumerate}
\end{remark}
}

\subsection{Asymptotic power}\label{subsec:power}
In this section we will analyze the power of our multiscale testing approach in the hypercube-case. The detection power clearly depends on the size and strength of the anomaly. To describe the latter, we will frequently employ the functions
\[
m(\theta):= \psi'(\theta)= \E{}{Y}, \qquad v(\theta):= \psi''(\theta)= \V{Y}
\]
for $Y \sim F_\theta$. 

\paragraph{Heuristics}
The key point for the following power considerations is that the observations in \eqref{eq:model} can be approximated as 
\begin{equation}\label{eq:signal_decomp}
{\frac{Y_i - m\left(\theta_0\right)}{\sqrt{v\left(\theta_0\right)}} =  \frac{m\left(\theta_i\right) - m\left(\theta_0\right)}{\sqrt{v\left(\theta_0\right)}} + \frac{\sqrt{v\left(\theta_i\right)}}{\sqrt{v\left(\theta_0\right)}} \frac{Y_i - m\left(\theta_i\right)}{\sqrt{v\left(\theta_i\right)}},}
\end{equation}
i.e. as 'signal' $v\left(\theta_0\right)^{-1/2} \left(m\left(\theta_i\right)- m\left(\theta_0\right)\right)$, which is non zero on the anomaly only, plus a standardized noise component $\left(Y_{i} - m\left(\theta_i\right)\right)/\sqrt{v\left(\theta_i\right)}$ which is scaled by a factor $v_i:= \sqrt{v\left(\theta_i\right)/v\left(\theta_0\right)}$. In case of Gaussian observations with variance $1$, one has $v_i \equiv 1$ and recovers the situation considered by \citet{sac16}. Whenever the 'signal' part in \eqref{eq:signal_decomp} is strong enough, the anomaly should be detected. In the following, we will make this statement mathematically precise and also give a comparison of the multiscale testing procedure with an oracle procedure.

\paragraph{Considered alternatives}

Consider a given family $\left(Q_n^*\right)_{n\in \mathbb N}$ of hypercube anomalies $Q_n^* \subset\left[0,1\right]^d$ with Lebesgue measure $\left|Q_n^*\right| = a_n \in \left(0,1\right)$. The corresponding discretized anomalies $Q_n:= I_n^d \cap nQ_n^* \subset I_n^d$ have size $\left|Q_n\right| \sim n^d a_n$. We will consider alternatives $K_{i,n}$ in \eqref{eq:multiple_testing} where $\theta^n \in \Theta^{n^d}$ s.t.
\begin{equation}\label{eq:anomaly}
\theta_i^n= \theta_1^n \mathbb{I}_{Q_n} {+\theta_0 \mathbb{I}_{Q_n^c}}.
\end{equation}
The parameters $\theta_1^n$ determine the total strength of the anomaly, which is given by
\[
\mu^n \left(Q_n\right) := \sqrt{\left|Q_n\right|} \frac{m\left(\theta_1^n\right) - m\left(\theta_0\right)}{\sqrt{v\left(\theta_0\right)}}.
\]
Clearly, any anomaly with fixed size or strength can be detected with asymptotic probability $1$. Therefore, we will consider vanishing anomalies in the sense that
\begin{equation}\label{eq:asymptotics}
a_n \searrow 0, \qquad \mu^n \left(Q_n\right) \nearrow \infty, \qquad\text{as}\qquad n \to \infty.
\end{equation}
Furthermore, we will restrict to parameters $\theta_1^n$ in \eqref{eq:anomaly} which yield uniformly bounded variances and uniform sub-exponential tails {for the standardized version}, this is
\begin{align}
\E{\theta}{\exp\left({s\frac{Y-m(\theta)}{\sqrt{v(\theta)}}}\right)} \leq & \,C \text{ for all } 0 \leq s \leq t \text{ and } \theta \in \left\{\theta_0\right\} \cup \bigcup\limits_{n \in \N} \left\{\theta_1^n\right\}, \label{eq:MGF}\\
{\underline{v} \leq} &{\sqrt{\frac{v\left(\theta_1^n\right)}{v\left(\theta_0\right)}} \leq \bar v \qquad\text{for all}\qquad n \in \mathbb N}\label{eq:variances}
\end{align}
for $Y\sim F_\theta$ with constants $t>0, C>0$ and $0 < \underline{v} < \bar v < \infty$.

In case of Gaussian observations with variance $\sigma^2$, \eqref{eq:MGF} and \eqref{eq:variances} are obviously satisfied, for a Poisson field this means that the intensities are bounded away from zero and infinity.

\paragraph{Oracle and multiscale procedure}

Recall that $\mathcal Q^*$ is the set of all hypercubes in $\left[0,1\right]^d$ (cf. Example \ref{ex:sets}), and $\mathcal Q^n$ its discretization (cf. \eqref{eq:defi_R_n}). 

If the size $a_n$ of the anomaly is known, but its position is still unknown, then one would naturally restrict the set of candidate regions to $\mathcal R^*_{\mathrm O} := \left\{Q^* \in \mathcal Q^* ~\big|~ \left|S^*\right| = a_n\right\}$, and consequently scan only over (cf. \eqref{eq:defi_R_n})
\[
\mathcal R_n^{\mathrm O} := \left\{Q \subset I_n^d ~\big|~ Q = I_n^d \cap nQ^* \text{ for some }Q^* \in \mathcal R^*_{\mathrm O} \right\}.
\]
As for the true anomaly $Q^* \in \mathcal R^*_{\mathrm O}$, its discretized version $Q_n$ also satisfies $Q_n \in \mathcal R_n^{\mathrm O}$. This gives rise to an oracle test, which rejects whenever $T_n\left(Y, \theta_0,\mathcal R_n^{\mathrm O},v\right)  > q_{1-\alpha,n}^{\mathrm O}$ where $q_{1-\alpha,n}^{\mathrm O}$ is the ${(}1-\alpha{)}${-}quantile of $M_n \left(R_n^{\mathrm O},v\right)$ as in \eqref{eq:Gaussian_approximation}. Similar as in Theorem \ref{thm:approximation} one can show that this quantile sequence ensures the oracle test to have asymptotic level $\alpha$. The asymptotic power of this oracle test can be seen as a benchmark for any multiscale test.

To obtain a competitive multiscale procedure, let us choose some $r_n$ satisfying the LSB \eqref{eq:r_n}, and furthermore assume that $r_n = o\left(n^d a_n\right)$, as otherwise the multiscale procedure will never be able to detect the true anomaly (as it is not contained in the set of candidate regions which we scan over). As now position and size of the anomaly are unknown, we consider all such sets in $\mathcal R^*_{\mathrm{MS}} = \mathcal Q^*$ as candidate regions and consequently scan over 
\[
\mathcal R_{n | r_n}^{\mathrm{MS}} := \left\{Q \subset I_n^d ~\big|~ Q = I_n^d \cap nQ^* \text{ for some }Q^* \in \mathcal Q^* \text{ and } \left|Q \right| \geq r_n \right\}.
\]
Clearly the true anomaly $Q^*$ satisfies $Q^* \in \mathcal R^*_{\mathrm{MS}}$, and by $r_n = o\left(n^d a_n\right)$ its discretized version $Q_n$ also satisfies $Q_n \in \mathcal R_{n| r_n}^{\mathrm{MS}}$. This gives rise to a multiscale test, which rejects whenever  $T_n\left(Y, \theta_0,\mathcal R_{n|r_n}^{\mathrm{MS}},v\right)  >q_{1-\alpha,n}^{\mathrm{MS}}$ where $q_{1-\alpha,n}^{\mathrm{MS}} := q_{1-\alpha}^{M_n\left(\mathcal R_{n | r_n}^{\mathrm{MS}},v\right)}$ is the ${(}1-\alpha{)}${-}quantile of $M_n \left(\mathcal R_{n| r_n}^{\mathrm{MS}},v\right)$ as in \eqref{eq:Gaussian_approximation}. Theorem \ref{thm:approximation} ensures that the multiscale test has asymptotic level $\alpha$.

Now, due to Theorem \ref{thm:weak_limit} $q_{1-\alpha}^* := q_{1-\alpha}^{M \left(\mathcal Q^*,v\right)} < \infty$ whenever {$v$ satisfies Assumption \ref{ass:v} (which corresponds to $v \geq 1$ here)}, it holds {that}
\[
q_{1-\alpha,n}^{\mathrm O} \leq q_{1-\alpha,n}^{\mathrm{MS}} \leq q_{1-\alpha}^* < \infty
\]
for all $n \in \mathbb N$.

\paragraph{Asymptotic power}

We will now show that the multiscale procedure described above (which requires no a priori knowledge on the scale of the anomaly) asymptotically detects the same anomalies with power $1$ as the oracle benchmark procedure for a known scale. Hence, the penalty choice to calibrate all scales as in \eqref{MultiscaleStat} (where $\mathcal R^* = \mathcal Q^*$), renders the adaptation to all scales for free, at least asymptotically. This can be seen as a structural generalization of \citep[Thms. 2 and 4]{sac16}, as under the alternative the whole distribution in \eqref{eq:model} and not just its mean might change. Also the power considerations in \citet{proksch2016multiscale} restrict to this simpler case. 

\begin{theorem}\label{thm:asymptPower}
In the setting described above, let $a_n \searrow 0 $ be a sequence of scales such that $\left(\log n\right)^{12}/n^d = o\left( a_n\right)$ as $n \to \infty$. Denote by 
\[
F \left(x,\mu,\sigma^2\right) := \Phi\left(-\frac{x + \mu}{\sigma}\right) + \Phi\left(\frac{\mu-x}{\sigma}\right), \qquad x \geq 0
\]
the survival function of a folded normal distribution with parameters $\mu \in \R$ and $\sigma^2>0$, where $\Phi$ is the cumulative distribution function of $\mathcal N \left(0,1\right)$. Let furthermore {$v \geq 1$}. If \eqref{eq:asymptotics} is satisfied, then the following holds true:
\begin{itemize}
\item[(a)] The single scale procedure has asymptotic power
\begin{multline*}
\P{\theta^n}{T_n\left(Y, \theta_0,\mathcal R_n^{\mathrm O},v\right)  > q_{1-\alpha,n}^{\mathrm O}}\\=\alpha + (1-\alpha)F\left(q_{1-\alpha,n}^{\mathrm O} + \sqrt{2v\log\left(\frac{1}{a_n}\right)},n^{d/2}\sqrt{a_n}\frac{m\left(\theta_1^n\right) - m\left(\theta_0\right)}{\sqrt{v\left(\theta_0\right)}},\frac{v\left(\theta_1^n\right)}{v\left(\theta_0\right)}\right) +o(1).
\end{multline*}
\item[(b)] If $a_n = o\left(n^{{\beta}-d}\right)$ with {$\beta>0$} sufficiently small, then the multiscale procedure has asymptotic power
\begin{multline*}
\P{\theta^n}{T_n\left(Y, \theta_0,\mathcal R_{n|r_n}^{\mathrm{MS}},v\right) >q_{1-\alpha,n}^{\mathrm{MS}}}\\\geq \alpha + (1-\alpha)F\left(q_{1-\alpha,n}^{\mathrm{MS}} + \sqrt{2v\log\left(\frac{1}{a_n}\right)},n^{d/2}\sqrt{a_n}\frac{m\left(\theta_1^n\right) - m\left(\theta_0\right)}{\sqrt{v\left(\theta_0\right)}},\frac{v\left(\theta_1^n\right)}{v\left(\theta_0\right)}\right) +o(1).
\end{multline*}
\end{itemize}
\end{theorem}

\begin{remark}
In \citep{sac16} a similar result in case of Gaussian observations is shown. {However, the proof of \citep[Thm. 4]{sac16} is incomplete and we require the additional condition that $a_n = o\left(n^{{\beta}-d}\right)$ with {$\beta>0$} sufficiently small for our proof.} 
In \citep{proksch2016multiscale} it suffices to assume $a_n \searrow 0$, as {large scales have been excluded s.t.} the maximum {tends to a} Gumbel-limit.  
\end{remark}

The above Theorem allows us to explicitly describe those anomalies which will be detected with asymptotic power $1$:
\begin{cor}\label{cor:power}
Under the setting in this section, the Assumptions of Theorem \ref{thm:asymptPower} {and if $v$ satisfies Assumption \ref{ass:v}}, any such anomaly is detected with asymptotic power $1$ either by the single scale or the multiscale testing procedure if and only if
\begin{align}\label{powercond}
\frac{\sqrt{2{v}\log\left(\frac{1}{a_n}\right)v\left(\theta_0\right)} - n^{d/2}\sqrt{a_n}\left|m\left(\theta_1^n\right) - m\left(\theta_0\right)\right|}{\sqrt{v\left(\theta_1^n\right)}} \to -\infty
\end{align}
as $n \to \infty$.
\end{cor}

{
\begin{remark}
\eqref{powercond} implies that a smaller value of $v$ makes more anomalies detectable. However, this is limited by Assumption \ref{ass:v}, which requires $v$ to be an upper bound of the complexity of $\mathcal Q^*$ in terms of the packing number. As we compute in Appendix \ref{appA}, this yields $v = 1$ as the optimal choice.
\end{remark}
}

\begin{example}
\begin{enumerate}
\item In case of Gaussian observations $Y_i \sim \mathcal N \left(\Delta_n \mathbb{I}_{Q_n},\sigma^2\right)$ with variance $\sigma^2$, where the baseline mean is $0$ and $\Delta_n$ the size of the anomaly, this yields detection if and only if
\[
\left|\Delta_n\right| n^{d/2}\sqrt{a_n} \succsim \sigma \sqrt{2 v \log\left(\frac{1}{a_n}\right)} \qquad\text{as}\qquad n \to \infty.
\]
If we calibrate the statistic with $v = 1$ (cf. Example \ref{ex:sets}), then this coincides with the well known asymptotic detection boundary for hypercubes, see e.g. \citet{CastroDonoho1,frick2014multiscale} for $d = 1$, \citet{bi13} for $d = 2$, or \citet{kou2017identifying} for general $d$.
\item For Bernoulli r.v.'s $Y_i \sim Ber\left(p_0 \mathbb I_{Q_n^c} + p_n\mathbb{I}_{Q_n}\right)$ with $p_0, p_n \in (0,1)$ s.t. $p_0+p_n\leq 1$, the condition \eqref{powercond} reads as follows:
\begin{align*}
\frac{\sqrt{2{v}p_0(1-p_0) \log\left(\frac{1}{a_n}\right)} - n^{d/2}\sqrt{a_n}\left|p_n - p_0\right|}{\sqrt{p_n(1-p_n)}} \to -\infty.
\end{align*}
Note, that the minimax detection rate is unknown in this case to best of our knowledge.
\item For a Poisson field $Y_i \sim \text{Poi} \left(\lambda_0 \mathbb I_{Q_n^c} + \lambda_n\mathbb{I}_{Q_n}\right)$ with $\lambda_0, \lambda_n>0$, Theorem \ref{thm:asymptPower} and Corollary \ref{cor:power} can only be applied if $\lambda_n$ is a bounded sequence. In this case, \eqref{powercond} reduces to
\begin{align*}
\frac{\sqrt{2{v}\lambda_0 \log\left(\frac{1}{a_n}\right)} - n^{d/2}\sqrt{a_n}\left|\lambda_n - \lambda_0\right|}{\sqrt{\lambda_n}} \to -\infty.
\end{align*}
Again, the minimax detection rate is unknown in this case to best of our knowledge.
\end{enumerate}
\end{example}

{
\section{Numerical simulations}\label{sec:numerics}

In this section we provide an implementation of the suggested multiscale testing procedure and discuss its computational complexity. Furthermore we explore the influence of the penalization parameter $v$ in \eqref{MultiscaleStat} on the finite sample power, the speed of convergence in \eqref{eq:weak_limit} and the influence of the LSB $r_n$ on the distribution of $T_n$ in \eqref{MultiscaleStat}. 

\subsection{Implementation and computational complexity}\label{sec:implementation}

To evaluate the statistic $T_n$ in \eqref{MultiscaleStat} in general, all local statistics $T_R$ have to be computed separately. Therefore, the computational complexity will in general be of the order $\mathcal O \left(\# \mathcal R_n \cdot n^d\right)$. Note that for the situations mentioned in Example \ref{ex:models}, each $T_R$ is given by a function of the local mean $\bar Y_R$, which already reduces the computational effort.

However, if the system of candidate regions $\mathcal R^*$ has a special convolution-type structure, a more efficient evaluation is possible. Therefore, assume that there is a global shape $B \subset I_n^d$ such that for every $R \in \mathcal R_n$ there exist $t,h \in I_n^d$ with $t_i + h_i \leq n$ for all $1 \leq i \leq d$ and $\mathbf 1_R \left(x\right) = \mathbf 1_B \left((x-t)/h\right)$. This is e.g. the case for the system of hyperrectangles or the system of hypercubes. In this special situation, we may use the fast Fourier transform (FFT). If we denote by $\ast$ a discrete convolution, then it holds that
\[
\bar Y_{R} = Y \ast \mathbf 1_{B} \left(\frac{\cdot -t}{h}\right) = \mathcal F^{-1} \left(\mathcal F \left(Y\right) \cdot \mathcal F\left(\mathbf 1_{B} \left(\frac{\cdot -t}{h}\right)\right)\right).
\]
Consequently, for a fixed scale $h$, all corresponding values $T_R$ can be computed by means of $3$ FFTs. Note that no zero-padding is necessary here as for an inverse problem \citep[see][]{proksch2016multiscale}. This gives a computational complexity of $\mathcal O \left(d\, \# \text{scales } n^d \log n\right)$ for a single evaluation of the test statistic $T_n$ in \eqref{MultiscaleStat}. In the hyperrectangle and hypercube case, using all possible scales, this yields $\mathcal O \left(dn^{2d} \log n\right)$ and $\mathcal O \left(dn^{d+1}\log n\right)$ respectively. Compared the naive implementation described at the beginning, which yield complexities $\mathcal O \left(n^{3d}\right)$ and $\mathcal O \left(n^{2d+1}\right)$, respectively, this is a significant improvement.

We also briefly mention a possible implementation using cumulative sums, which is also possible for hyperrectangles and hypercubes. Once the cumulative sum of all observations has been computed, each local mean $\bar Y_R$ can be computed summing or subtracting $2^d$ values. Hence, this implementation gives in in general a computational complexity of $\mathcal O \left(n^d + 2^d \cdot \# \mathcal R_n\right)$, which yields $\mathcal O \left(2^d n^{2d}\right)$ and $\mathcal O \left(2^dn^{d+1}\right)$ for hyperrectangles and hypercubes respectively. Compared to the implementation using FFT described above, this differs by a factor $d 2^{-d} \log n$, which reveals the FFT implementation to be more efficient for large $d$. 

Note that in many applications, a priori information is available, which allows to select a (small) subset of scales instead of using all, which clearly reduces the computational effort further.

We emphasize that the quantiles $q_{1-\alpha,n}$ of the approximating Gaussian version \eqref{eq:Gaussian_approximation} can be universally pre-computed and stored as long as $n$ and the system $\mathcal R_n$ do not change, and for large $n$ the asymptotic values can be used in a universal manner (cf. Subsection \ref{subsec:asympt} below). Even for small values of $\alpha$, the above implementation allows to simulate the $(1-\alpha)$-quantile of the Gaussian approximation \eqref{eq:Gaussian_approximation} efficiently. This makes fast computations on incoming data sets in an 'online' fashion possible, which is important in many applications. In contrast, permutation based methods as considered in \citep{arias2017distribution} require to simulate the unknown null distribution separately for every given problem instance. 

\subsection{Influence of $v$ on the power}\label{subsec:v}

To study the influence of the penalization parameter $v$ on the power of the procedure, we turn to the setting of Section \ref{subsec:power}. Let $n = 512$ and $d = 2$. For simplicity, we consider a Gaussian model, i.e. $F_\mu = \mathcal N \left(\mu,1\right)$ in \eqref{eq:model} and choose $\mu_i= \mu \mathbb{I}_{Q}$ with $\mu \in \left\{1,1.2\right\}$ and $\left|Q\right| \in \left\{6^2,7^2\right\}$. Afterwards, we simulate the empirical power from 1000 repetitions. This procedure is performed for the VC-based choice $v = 3$ and for the capacity-based choice $v = 1$, which is asymptotically minimax optimal (see Corollary \ref{cor:power}). The results are depicted in Table \ref{tab:power}.
\begin{table}[!htb]
\centering
\begin{tabular}{cc}
\begin{tabular}{c|c|c}
\toprule[2pt]
$\left|Q\right|$ and $\mu$ & $1$ & $1.2$ \\
\midrule[1pt]
$6^2$ & $0.429$ & $0.817$\\
\midrule[1pt]
$7^2$ & $0.809$ & $0.983$\\
\bottomrule[2pt]
\end{tabular}
&
\begin{tabular}{c|c|c}
\toprule[2pt]
$\left|Q\right|$ and $\mu$ & $1$ & $1.2$ \\
\midrule[1pt]
$5^2$ & $0.104$ & $0.182$\\
\midrule[1pt]
$6^2$ & $0.187$ & $0.577$\\
\bottomrule[2pt]
\end{tabular} \\[0.2cm]
$v = 1$ & $v = 3$
\end{tabular}
\caption{Empirical power of the investigated testing procedure for different choices of $v$ in different Gaussian settings determined by $\mu$ (columns) and $\left|Q\right|$ (rows).}
\label{tab:power}
\end{table}

We find that the power for $v = 1$ is substantially larger than the one obtained by using the VC-dimension for calibration. This is in line with our findings from Example \ref{ex:cub_rect}.

\subsection{Speed of convergence in \eqref{eq:weak_limit}}\label{subsec:asympt}

To investigate the speed of convergence in \eqref{eq:weak_limit}, we consider the system of hypercubes $\mathcal R_n = \mathcal Q_n$ as in Subsection \ref{subsec:v}. Figure \ref{fig:densities} shows estimated densities of $M_n$ for different values of $n$ in dimensions $d = 1$ and $d = 2$. 

\begin{figure}[!htb]
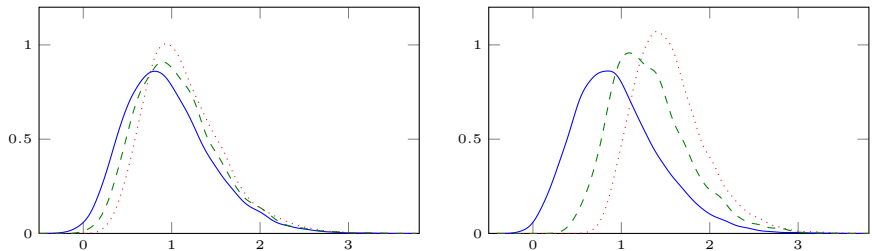

\setlength\fheight{3cm} \setlength\fwidth{5cm}
\centering
\tiny
\begin{tabular}{cc}
\input{densities_1d.tikz} & \input{densities_2d.tikz}
\end{tabular}
\caption{Simulated densities of the Gaussian approximations, displayed by a standard kernel estimator obtained from $10^4$ runs of the test statistic \eqref{eq:Gaussian_approximation}, with different values of $n$ and $d$.
Left: $d = 1$ and $M_n \left(\mathcal Q_n,1\right)$ with $n = 2^{10}$ (\ref{10}), $n = 2^{12}$ (\ref{12}) and $n = 2^{14}$ (\ref{14}).
Right: $d = 2$ and $M_n \left(\mathcal Q_n,1\right)$ with $n = 2^{5}$ (\ref{5}), $n = 2^{7}$ (\ref{7}) and $n = 2^{9}$ (\ref{9}).}
\label{fig:densities}
\end{figure}

We find that the speed of convergence of $M_n$ towards the weak limit $M$ in \eqref{eq:Gaussian_limit} decreases with increasing $d$, but we can however conclude that the distribution of $M_n$ stabilizes already at moderate values of $n$. This is especially helpful in situations, where data with significantly larger sample size $n$ is given, such that the distribution of $T_n$ cannot be simulated anymore.

\subsection{Influence of the lower scale bound $r_n$}

Let us again consider $n = 512$, $d = 2$ and the system of hypercubes $\mathcal R_n = \mathcal Q_n$ as in Subsection \ref{subsec:v}. Let $Y_i \sim \text{Bin}\left(1,\theta\right)$ and $\theta_0 = 0.5$. Figure \ref{fig:r_n} shows the simulated densities of $T_n \left(Y, \theta_0, \mathcal R_{n | r_n}, v\right)$ for different values of $r_n$.

\begin{figure}[!htb]
\setlength\fheight{3cm} \setlength\fwidth{6cm}
\centering
\tiny
\input{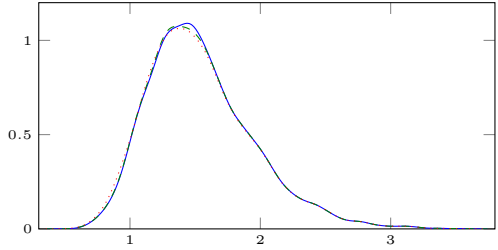} 
\caption{Simulated densities of the test statistic $T_n \left(Y, \theta_0, \mathcal R_{n | r_n}, v\right)$ in \eqref{MultiscaleStat} in case of i.i.d. Bernoulli observations with $p = 1/2$ for different values of $r_n$: $r_n = 2^3$ (\ref{23}), $r_n = 2^4$ (\ref{24}), $r_n = 2^5$ (\ref{25}). }
\label{fig:r_n}
\end{figure}

In conclusion, we find that the distribution of $T_n \left(Y, \theta_0, \mathcal R_{n | r_n}, v\right)$ is surprisingly robust w.r.t.\ the choice of $r_n$ even below the LSB \eqref{eq:r_n}.
}

\section{Auxiliary results}

In this section, we will present the main ingredients needed for our proofs, which might be of independent interest. One tool is a coupling result which allows us to replace the maximum over partial sums of standardized NEF r.v.'s by a maximum over a corresponding Gaussian version. This can be obtained from recent results by \citet{chernozhukov2014gaussian} as soon as certain moments can be controlled, which is the purpose of the following two {lemmas, which generalize known bounds for sub-Gaussian random variables to sub-exponential ones}. In what follows, the letter $C>0$ denotes some constant, which might change from line to line. 

{The following lemma gives an upper bound for the maximum of} uniformly sub-exponential random variables:
\begin{lem}\label{prop.mean_of_max_potenz_general}
Let $W_i,  ~i=1,2, \ldots$ be independent sub-exponential random variables s.t. there exist $k_1>1$ and $ k_2 >0$ s.t. 
\begin{align}\label{eq.glmSubExpTails}
\P{}{|W_i|>t} \leq k_1 \exp(-k_2 t)
\end{align} for all $i$. 
Then for all $m \in \mathbb{N}$ there exists a constant $C,$ s.t. for all $N\geq 2$ 
\begin{align*}
\E{}{\max\limits_{1\leq i \leq N} |W_i|^m} \leq C \left(\log N\right)^m.
\end{align*}
\end{lem}
{Lemma \ref{prop.mean_of_max_potenz_general} might be of independent interest, as it generalizes the well known bound} 
\begin{equation}\label{eq:Gaussian_max}
\E{}{\max\limits_{1\leq i \leq N} |X_i|} \leq C \sqrt{\log N}
\end{equation}
for sub-Gaussian random variables {to sub-exponential random variables}.

Now we will show that the maximum over the partial sum process of independent random variables can be bounded by the maximum over the corresponding Gaussian version. The latter can be controlled as in \eqref{eq:Gaussian_max} by exploiting the fact that a maximum over dependent Gaussian random variables is always bounded by a maximum over corresponding independent Gaussian random variables \citep[see e.g.][]{vsidak1967rectangular} 
\begin{equation}\label{eq:Gaussian_partial_sums_max}
\E{}{ \max\limits_{I \in \mathcal{I}} \frac{\left|X_I \right|}{\sqrt{|I|}}} \leq C \sqrt{\log\left(\#\left(\mathcal I\right)\right)}
\end{equation}
with $X_i \stackrel{\text{i.i.d.}}{\sim} \mathcal N \left(0,1\right)$ and $X_I:= \sum\limits_{i \in I} X_i$. This allows us to prove the following:
\begin{lem}\label{prop:partialsum} 
Let $(Z_i)_{i=1, \ldots, N}$ be independent random variables with $\E{}{Z_i} =0$ and denote $Z_I:= \sum\limits_{i \in I} Z_i$. If $\mathcal{I}$ is an arbitrary index set of sets $\{I\}_{I \in \mathcal{I}}$, then there exists a constant $C>0$ independent of $\mathcal{I}$ s.t.
\begin{align*}
\E{}{ \max\limits_{I \in \mathcal{I}} \frac{\left|Z_I \right|}{\sqrt{|I|}}} \leq C \sqrt{\log\left(\#\left(\mathcal I\right)\right)} ~\E{}{ \max\limits_{1 \leq i \leq N} \left| Z_i\right| }.
\end{align*}
\end{lem}

\begin{theorem}[Coupling]\label{lemma:Chernozhukov_NEF}
Let $Z_i, i \in I_n^d$ independent, $\E{}{Z_i}=0, \V{Z_i}=1$, such that \eqref{eq.glmSubExpTails} is satisfied for all $i$ with uniform constants  $k_1>1$ and $ k_2 >0$. Let furthermore {$a_i, i \in I_n^d$ with $0 < \inf a_i \leq \sup a_i < \infty$} independent of $i$ and $n$, and $X_i \overset{i.i.d.}{\sim} \mathcal{N}(0,1), i=1, \ldots, n^d,$ and $\mathcal R_n$, s.t. inequality \eqref{eq:finite_cardinality} holds. 
Then
\begin{align*}
\max\limits_{\substack{R \in \mathcal R_n:\\ |R| \geq r_n}}{|R|^{-1/2}}  \sum_{i \in R} {a_i} Z_i- \max\limits_{\substack{R \in \mathcal R_n:\\ |R| \geq r_n}} |R|^{-1/2} \sum_{i \in R} {a_i} X_i = \OP \left(\left( \frac{\log^{10}(n)}{r_n}\right)^{1/6} \right).
\end{align*}
\end{theorem}

{
\begin{remark}
Note that Theorem \ref{lemma:Chernozhukov_NEF} requires only $\log^{10} (n) = o(r_n)$ for convergence in probability, whereas we require an exponent of $12$ in the LSB \eqref{eq:r_n}. The reason is, that Theorem \ref{lemma:Chernozhukov_NEF} yields a coupling for the unpenalized partial sums, whereas Theorems \ref{thm:approximation} and \ref{thm:weak_limit} work with penalized partial sums. Including the penalty term requires an additional slicing argument, which results in an additional $2$ in the exponent (see the proof of Theorem 5 in the supplement).
\end{remark}
}

\section{Proofs}

In this section we will give all proofs. In the following we will denote by $p_n$ the cardinality of $\mathcal R_n$, i.e. $p_n:= \#(\mathcal R_n)$, which by \eqref{eq:finite_cardinality} satisfies $\log(p_n) \sim \log n$. Recall that $C$ denotes a generic constant which might differ from line to line.

\subsection{Proof of the auxiliary results}

We start with proving the auxiliary statements from section {2}.

\begin{proof}[Proof of Lemma \ref{prop.mean_of_max_potenz_general}]
Let $h(t):=k_1 \exp(-k_2 t)$, then
\begin{align*}
\P{}{\max\limits_{1 \leq i \leq N} |W_i| >t} &= 1- \P{}{\max\limits_{1 \leq i \leq N} |W_i| \leq t} \leq 1 - \left( 1- h(t)\right)^N\leq Nh(t).
\end{align*}
Let $ \bar{t}= h^{-1}(1/N){\sim} C \log(N)$, then
\begin{align*}
\E{}{\max\limits_{1 \leq i \leq N} |W_i|^m }&= m \int_0^\infty t^{m-1} \P{}{\max\limits_{1 \leq i \leq N} |W_i| >t} \,\mathrm dt\\
&\leq m \int_0^{\bar{t}} t^{m-1}\,\mathrm dt + m \int_{\bar{t}}^\infty t^{m-1}N h(t) \,\mathrm dt\\
&\leq \left( C \log(N) \right)^m + k_1mN \int_{\bar{t}}^\infty t^{m-1} \exp(-k_2 t)\,\mathrm dt\\
&\leq C \left( \log N\right)^m,
\end{align*}
where the last inequality follows from integration by parts.

\end{proof}

\begin{proof}[Proof of Lemma \ref{prop:partialsum}]
Let $X_i \stackrel{\text{i.i.d.}}{\sim} \mathcal N \left(0,1\right)$ and $r_i$ be i.i.d. Rademacher random variables, i.e. they take the values $\pm 1$ with probability $1/2$.\\
\textit{Step (i): }
Since the $X_i$ are symmetric 
\begin{align}\label{ineq.symmetric}
\E{}{\max\limits_I \frac{1}{\sqrt{|I|}} \left| \sum\limits_{i \in I} X_i\right|} = \E{r}{\E{}{\max\limits_I \frac{1}{\sqrt{|I|}} \left| \sum\limits_{i \in I} r_i X_i\right|}}.
\end{align}
By Lemma 4.5 of \citep{ledoux1991probability} and choosing $F(t)=t, ~\eta_i = X_i$ and $x_i:= \left(c_{i,I}\right)_I$, where $c_{i,I}:=  \frac{1}{\sqrt{|I|}} \mathbb{I}_{\{i \in I\}}$ (a scaled indicator function) and as norm the $\max-$norm, we obtain
\begin{align}\label{eq: RademacherMeanEstimation}
\E{r}{\max\limits_I \frac{1}{\sqrt{|I|}} \left| \sum_{i \in I} r_i \right|} \leq \sqrt{\frac{\pi}{2}} \E{}{\max\limits_I \frac{1}{\sqrt{|I|}} \left|\sum_{i \in I} X_i\right|}.
\end{align}
\textit{Step (ii):} 
Let $(Z_i')_{1 \leq i \leq N}$ be a sequence of independent copies of $(Z_i)_{1 \leq i \leq N}$ and define the symmetrized version of $Z_i$ by $\widetilde{Z_i}:= Z_i-Z_i'$ and equally the symmetrized version of $Z_I$ by $\widetilde{Z_I} := \sum_{i \in I} (Z_i - Z_i')$. Then by using the same argument as in \eqref{ineq.symmetric} and Fubini's theorem, we derive
\begin{align*}
\E{}{\max\limits_I \frac{1}{\sqrt{|I|}} \left| Z_I\right|} &\leq 2\E{}{\max\limits_I \frac{1}{\sqrt{|I|}} \left| \widetilde{Z_I}\right|}\\
&=2 \E{r}{\E{}{\max\limits_I \frac{1}{\sqrt{|I|}} \left| \sum_{i \in I} \widetilde{Z_i} r_i\right|}}\\
&=2 \E{}{\E{r}{\max\limits_I \frac{1}{\sqrt{|I|}} \left| \sum_{i \in I} \widetilde{Z_i} r_i\right|}}\\
&= 2\E{}{\E{r}{\max\limits_I \frac{1}{\sqrt{|I|}} \left| \sum_{i \in I} \left| \widetilde{Z_i}\right| r_i\right|}},
\end{align*}
where the last equality holds in view of the symmetry of $r_i$. Now we will use the contraction principle, i.e. Theorem 4.4 of  \citep{ledoux1991probability} with $F(t)=t$ conditionally on $\alpha_i :=\frac{\left| \widetilde{Z_i}(\omega)\right|}{\max\limits_j | \widetilde{Z_j} (\omega)|}$, which is independent of $(r_i)$. 
By choosing $x_i$ as in \textit{Step (i)} we get (after multiplying both sides with $\max\limits_j |\widetilde{Z_j}(\omega)|$)
\begin{align*}
\E{r}{\max\limits_I \frac{1}{\sqrt{|I|}} \left| \sum_{i \in I} \left| \widetilde{Z_i}(\omega)\right|r_i\right|} \leq \E{}{\max\limits_{1\leq i  \leq N} \left|\widetilde{Z_i}(\omega) \right|} \E{r}{\max\limits_I \frac{1}{\sqrt{I}} \left| \sum_{i \in I} r_i\right|}.
\end{align*}
Therefore 
\begin{align*}
\E{}{\max\limits_I \frac{1}{\sqrt{|I|}} \left| Z_I\right|} &\leq 2\E{}{\max\limits_{1 \leq i \leq N} \left| \widetilde{Z_i}(\omega)\right|} \E{r}{\max\limits_I \frac{1}{\sqrt{|I|}} \left| \sum_{i \in I} r_i\right|}\\
&\leq 4 \E{}{\max\limits_{1 \leq i \leq N} \left| Z_i\right|} \sqrt{\frac{\pi}{2}}\E{}{\max\limits_I \frac{1}{\sqrt{|I|}} \left| \sum_{i \in I} X_i\right|}\\
&= \sqrt{8\pi} \E{}{\max\limits_{1 \leq i \leq N} \left| Z_i \right|} \E{}{\max\limits_I \frac{\left| X_I\right|}{\sqrt{|I|}}},
\end{align*}
where we used \eqref{eq: RademacherMeanEstimation} in the second inequality. Now the statement follows from \eqref{eq:Gaussian_partial_sums_max}.
\end{proof}

\begin{proof}[Proof of Theorem \ref{lemma:Chernozhukov_NEF}]
Enumerate each region in $\mathcal R_n$ by $j, 1 \leq j \leq p_n$ and define
\begin{align}
\begin{aligned}\label{def.Xij}
X_{ij} &:= \frac{{a_i}}{\sqrt{|R_j|}} Z_i \mathbb{I}_{\{i \in R_j\}} \mathbb{I}_{\{|R_j| \geq r_n\}},\\
X_i &:= \left( X_{ij}\right)_{j=1, \ldots, p_n}, \quad i=1, \ldots, N= n^d,
\end{aligned}
\end{align}
for some sequence $r_n$.
Then $Z:= \max\limits_{1 \leq j \leq p_n} \sum_{i=1}^N X_{ij}$ satisfies
\[
Z \stackrel{\mathcal D}{=} \max\limits_{\substack{R \in \mathcal R_n:\\ |R| \geq r_n}} \frac{1}{\sqrt{|R|}} \sum_{i \in R} {a_i} Z_i.
\]
Recall that $\log(p_n) \lesssim \log(n)$. According to \citep[][Cor. 4.1]{chernozhukov2014gaussian} we find that for every $\delta>0$ there exists a Gaussian version $\tilde{Z}:= \max\limits_{1 \leq j \leq p_n} \sum_{i=1}^N {a_i} N_{ij}$ with independent random vectors $N_1, \ldots, N_n$ in $\R^{p_n}$, $N_i \sim N(0, \E{}{X_i X_i^t}), 1\leq i \leq N$, such that
\begin{align*}
\P{}{\left| Z - \tilde{Z}\right|> 16 \delta} \lesssim \delta^{-2} \left\lbrace B_1 + \delta^{-1} (B_2+B_4) \log(n) \right\rbrace \log(n) + \frac{\log(n)}{n^d}
\end{align*}
where
\begin{align*}
B_1 &:= \E{}{\max\limits_{1\leq j, k \leq p_n} \left| \sum_{i=1}^N \left( X_{ij}X_{ik}-\E{}{X_{ij}X_{ik}}\right)\right|}\\
B_2 &:= \E{}{\max\limits_{1 \leq j \leq p_n} \sum_{i=1}^N \left| X_{ij}\right|^3}\\
B_4&:= \sum_{i=1}^N \E{}{\max\limits_{1 \leq j \leq p_n} \left| X_{ij}\right|^3 \mathbb{I}_{\{\max\limits_{1 \leq j \leq p_n} \left| X_{ij}\right| > \delta/ \log(p_n \vee n)\}}}.
\end{align*}
$B_1$ can be controlled as follows. With $X_{ij}$ from \eqref{def.Xij} we derive
\begin{align*}
B_1 &= \E{}{ \max\limits_{\substack{1 \leq j,k \leq p_n:\\ |R_j|, |R_k| \geq r_n}} \left| \sum_{i \in R_j \cap R_k}  \frac{{a_i^2}(Z_i^2-1)}{\sqrt{|R_j| |R_k|}}\right| }\\
&=\E{}{ \max\limits_{\substack{1 \leq j,k \leq p_n:\\ |R_j|, |R_k| \geq r_n}}\frac{\sqrt{|R_j \cap R_k|}}{\sqrt{|R_j| |R_k|}} \left| \frac{1}{\sqrt{|R_j \cap R_k|}} \sum_{i \in R_j \cap R_k}  {a_i^2}(Z_i^2-1) \right|  }.
\end{align*}
Using the restriction on the size of the rectangles we find: 
\begin{align*}
\sqrt{\frac{|R_j \cap R_k|}{|R_j| |R_k|}} \leq \sqrt{\frac{|\min\{|R_j|, |R_k|\}|}{|R_j| |R_k|}} \leq \frac{1}{\sqrt{r_n}}.
\end{align*}
Denote $V_i := {a_i^2}(Z_i^2 -1), ~I:= R_j \cap R_k \in \mathcal{I} \subset I_n^d$ and $	 S_I := \sum_{i \in I} V_i$.
Now
\begin{align*}
B_1 \leq \frac{1}{\sqrt{r_n}} \E{}{ \max\limits_{I \in \mathcal{I}} \frac{\left| S_I \right|}{\sqrt{|I|}}}.
\end{align*}
Using Lemma \ref{prop:partialsum} we obtain
\begin{align*}
B_1 \leq \frac{C}{\sqrt{r_n}} \underbrace{\sqrt{\log\left(\#(\mathcal{I}) \right)}}_{\sim \sqrt{\log(n)}} ~\E{}{\max\limits_{1 \leq i \leq N} |{a_i^2}(Z_i^2-1)| } 
\end{align*}
It remains to estimate
\begin{align*}
\E{}{\max\limits_{1 \leq i \leq N} |{a_i^2}(Z_i^2-1)| }  &= \E{}{\max\limits_{1 \leq i \leq N} {a_i^2}|Z_i^2-1| } \\
&\leq \E{}{\max\limits_{1 \leq i \leq N} {\overline{a}^2}|Z_i^2-1| }\\
&\leq {\overline{a}^2} \E{}{\max\limits_{1 \leq i \leq N} \left| Z_i \right|^2} +{\overline{a}^2}{,}
\end{align*} 
{where $\overline{a}:= \sup a_i.$}
So in total we get by Lemma \ref{prop.mean_of_max_potenz_general}
\begin{align*}
B_1 \lesssim \frac{\sqrt{\log(n)}}{\sqrt{r_n}} \left( {\overline{a}^2} C \log(N)^2 + {\overline{a}^2} \right)\lesssim  \left( \frac{\log^5 (n)}{r_n} \right)^{1/2}.
\end{align*}
For $B_2$ we compute
\[
\begin{aligned}
B_2 \leq \frac{1}{\left( r_n\right)^{1/2}} \E{}{\max\limits_{1 \leq i \leq N} \left| {a_i}Z_i\right|^3 }&\\
\leq \frac{{\overline{a}^3}}{\left( r_n\right)^{1/2}} & \E{}{\max\limits_{1 \leq i \leq N} \left|Z_i\right|^3 }
\lesssim \left( \frac{\log^6(n)}{r_n}\right)^{1/2},
\end{aligned}
\]
where we again used Lemma \ref{prop.mean_of_max_potenz_general}. Now let $\delta>0$ be fixed. Then
\begin{align*}
B_4 &\leq \sum\limits_{i=1}^N \E{}{ \max\limits_{\substack{1 \leq j \leq p_n:\\ |R_j|\geq r_n}} \frac{|{a_i} Z_i|^3}{|R_j|^{3/2}} \mathbb{I}_{\{\max\limits_j \frac{|{a_i} Z_i|}{|R_j|^{1/2}}> \frac{\delta}{\left( {\log p_n}\right)}\}} }\\
&\leq \frac{N}{r_n^{3/2}} \max\limits_{1 \leq i \leq N} \E{}{ |{a_i} Z_i|^3 \mathbb{I}_{\{|Z_i| > \frac{\delta r_n^{1/2}}{|{a_i}| {\log p_n}}\}} }
\end{align*}
Now let $r_n > \left(\frac{2d |{a_i}|}{\delta} \right)^2 (\log n)^{2+2\gamma}$ for some $\gamma>1$ for $n \geq n_o(\delta, d)$ and hence
\begin{align*}
B_4 &\leq \frac{N}{r_n^{3/2}}\max\limits_{1 \leq i \leq N} \E{}{  |{a_i}|^3 |Z_i|^3 \mathbb{I}_{\{|Z_i| > \left(\log n \right)^\gamma\}}}\\
&\leq  \frac{N {\overline{a}^3}}{r_n^{3/2}} 3 \max\limits_{1 \leq i \leq N} \int_{(\log n)^\gamma}^\infty t^2 \P{}{  |Z_i| > t } \,\mathrm dt\\
&\leq 3 k_1 \frac{N {\overline{a}^3}}{r_n^{3/2}} \int_{(\log n)^\gamma}^\infty t^2 \exp(- k_2 t) \,\mathrm dt\\
&= 3k_1 \frac{N {\overline{v}^3}}{r_n^{3/2}} \frac{1}{ k_2^{3}} \int_{k_2 (\log n)^{\gamma }}^\infty u^{2} \exp(-u) \,\mathrm du.
\end{align*}
For $u$ large enough s.t. $u^2 \leq \exp(u/2)$, i.e. for $ n \geq n_1(\delta, d, \gamma)$, it holds
\[
B_4 \leq \frac{3 k_1 {\overline{a}^3}}{k_2^3} \frac{N}{r_n^{3/2}} \int_{k_2 \left(\log n\right)^{\gamma}}^\infty \exp\left(- \frac{u}{2}\right) \,\mathrm du= \frac{3 k_1 {\overline{a}^3}}{k_2^3} \frac{N}{r_n^{3/2}} \exp\left( -\frac{k_2}{2} \left( \log n\right)^{\gamma }\right),
\]
and then furthermore $\frac{k_2}{2} \left( \log n\right)^{ \gamma} \geq \left(d \log n \right)$ which implies
\[
B_4 \leq \frac{3 k_1 {\overline{a}^3}}{k_2^3} \frac{n^d}{\left(r_n\right)^{3/2}} n^{-d} = \frac{3 k_1 {\overline{a}^3}}{k_2^3} \frac{1}{\left( r_n\right)^{3/2}}.
\]
In conclusion we obtain
\begin{align*}
\P{}{|Z-\tilde{Z}|> 16 \delta} \lesssim & \delta^{-2}\left( \frac{\log^7(n)}{r_n}\right)^{1/2} + \delta^{-3}\left( \frac{\log^{10}(n)}{r_n}\right)^{1/2} \\
& + \delta^{-3}\left( \frac{\log^4(n)}{r_n^3}\right)^{1/2} + \frac{\log(n)}{n^{d}},
\end{align*}
which yields the claim.
\end{proof}

\subsection{Proofs of Section \ref{subsec:limit}}
Let us now prove the results from Section \ref{subsec:limit}, including Theorems \ref{thm:weak_limit} and \ref{thm:approximation}. We start with a Taylor expansion of $T_n$, which will allow us to apply Theorem \ref{lemma:Chernozhukov_NEF}. 

\begin{lem}\label{prop.approx.hyperrectangles} 
Let $\mathcal R_n$ be a collection of sets s.t. \eqref{eq:finite_cardinality} holds, $\epsilon>0$ and $(r_n)_n\subset(0,\infty)$ be a sequence, s.t.\ $(\log n)^{10 + \epsilon}/r_n\rightarrow0$. Suppose $Y_{i}\sim F_{\theta_0} \in \mathcal{F}$, $i \in I_n^d$, are i.i.d. random variables, and recall that for $R \in \mathcal R_n$ we denote $\overline{Y}_R := \left|R\right|^{-1} \sum_{i \in R} Y_i$. Then it holds {that}
\[
\max_{\substack{R\in\mathcal R_n:\\|R|\geq r_n}}\left|T_R(Y,\theta_0)-|R|^{\frac{1}{2}}\frac{|\overline{Y}_R-m(\theta_0)|}{\sqrt{v(\theta_0)}}\right|= \OP\left( \left( \frac{\log^3 (n)}{r_n}\right)^{1/4}\right)
\]
as $n \to \infty$.
\end{lem}
\begin{proof}
For independent Gaussian random variables $ X_i \sim \mathcal N \left(0,1\right)$ it follows from {\eqref{eq:Gaussian_partial_sums_max}} and\eqref{eq:finite_cardinality} that
\begin{align*}
\E{}{ \left|\max\limits_{R \in \mathcal R_n} |R|^{-1/2} \sum\limits_{i \in R} X_i \right|} \leq C\sqrt{\log n}, 
\end{align*}
hence
\[
\frac{1}{\sqrt{\log(n)}} \max\limits_{\substack{R \in \mathcal R_n:\\|R| \geq r_n}} |R|^{-1/2} \sum\limits_{i \in R} X_i = \oP(1).
\]
Combining this result with Theorem \ref{lemma:Chernozhukov_NEF} (with {$a_i=1$} for all $i$) we obtain
\begin{equation}\label{eq:ReplacementOfKM}
\frac{1}{\sqrt{\log n}} \max\limits_{\substack{R \in \mathcal R_n:\\|R| \geq r_n}} \frac{1}{\sqrt{|R|}} \left| \sum_{i \in R} \frac{Y_i - m(\theta_0)}{\sqrt{v(\theta_0)}}\right| = \oP(1).
\end{equation}
Together with \eqref{eq:finite_cardinality} it follows
\begin{align*}
\max_{\substack{R\in\mathcal R_n:\\|R|\geq r_n}}\left|\overline{Y}_R-m(\theta_0)\right|\leq C\left(\frac{ \log(n) v(\theta_0)}{r_n}\right)^{\frac12}(1+\oP(1))\rightarrow0, \quad n \rightarrow \infty.
\end{align*} 
Therefore, $\overline{Y}_R>m(\theta_0)/\sqrt2$ in probability if $n$ is large enough uniformly over  $R$, s.t. $|R|\geq r_n$. {Let $\phi \left(x\right) := \sup_{\theta \in \Theta} \left[\theta \cdot x - \psi \left(\theta\right)\right]$ be the Legendre-Fenchel conjugate  of $\psi$ and 
\[
J \left(x, \theta\right) := \phi\left(x\right) - \left[ \theta \cdot x - \psi \left(\theta\right)\right],
\]
then the LRT statistic $T_R(Y,\theta_0)$ in \eqref{eq:lr_stat} can be written as
\begin{align}
T_R(Y,\theta_0)
=&\sqrt{2 \left(\sup\limits_{\theta} \sum\limits_{i \in R} \left(\theta \cdot Y_i- \psi(\theta) \right) - \sum\limits_{i \in R} \left(\theta_0 \cdot Y_i - \psi(\theta_0) \right)\right)}\nonumber\\
=& \sqrt{2 \left|R\right| J \left(\bar Y_R, \theta_0\right)}\label{eq:lr_stat_representation}
\end{align}
with $\bar Y_R=|R|^{-1}\sum_{i\in R}Y_{i}$. Note that by definition it holds $ J \left(\bar Y_R, \theta_0\right) \geq 0$.} 
As the $\sup\limits_{\theta \in \Theta} \prod\limits_{i \in R} p_\theta(Y_{i})$ is attained at that $\theta$ for which  $\psi'(\theta) = \overline{Y}_R$ we derive
\begin{align*}
\phi(\overline{Y}_R) = \langle m^{-1}(\overline{Y}_R), \overline{Y}_R\rangle - \psi\left(m^{-1}(\overline{Y}_R)\right),
\end{align*}
and therefore
\begin{align*}
J(\overline{Y}_R, \theta_0) &= \langle m^{-1}(\overline{Y}_R), \overline{Y}_R\rangle - \psi\left(m^{-1}(\overline{Y}_R)\right) - \left( \langle \theta_0, \overline{Y}_R \rangle - \psi(\theta_0)\right)\\
&= \langle m^{-1}(\overline{Y}_R)- \theta_0, \overline{Y}_R \rangle - \left( \psi\left( m^{-1}(\overline{Y}_R) \right)-\psi(\theta_0)\right).
\end{align*}
Note that $\bar{Y}_R \in D(m^{-1})$ for large enough $n$, as the latter is an open set. A Taylor expansion of $\psi$ around $\theta_0$ and one of second order of $m^{-1}$ around $m(\theta_0),$ yields
\begin{align}\label{eq.Taylor_general_non_stand}
T_R(Y, \theta_0) &= {\left(|R| \left( \frac{\overline{Y}_R- m(\theta_0)}{\sqrt{v(\theta_0)}}\right)^2 + |R| s_n\left( \frac{\overline{Y}_R- m(\theta_0)}{\sqrt{v(\theta_0)}} \right)\right)^{1/2}}
\end{align}
with $s_n$ s.t.\ $\left| s_n\left(x \right)\right| \leq c x^3 + \oP(1)$ for some $c>0$. Consequently
\begin{align*}
&\max_{\substack{R\in\mathcal R_n:\\|R|\geq r_n}}\left|{
T_R^2(Y,\theta_0)}-|R|\frac{(\overline{Y}_R-m(\theta_0))^2}{v(\theta_0)}\right| \\ 
\lesssim &\max_{\substack{R\in\mathcal R_n:\\|R|\geq r_n}}|R|\frac{|\overline{Y}_R-m(\theta_0)|^3}{v(\theta_0)^{3/2}} + \oP(1)\\
=&\max_{\substack{R\in\mathcal R_n:\\|R|\geq r_n}}|R|^{-\frac{1}{2}}\left|\frac{\sum_{i\in R}(Y_{i}-m(\theta_0))}{\sqrt{|R|}\sqrt{v(\theta_0)}}\right|^3 + \oP(1)\\[0.1cm]
\leq & (\log^3(n)r_n^{-1})^{1/2} + \oP(1),
\end{align*}
where we again used \eqref{eq:ReplacementOfKM}. Now $|a-b|\leq|a^2-b^2|^{\frac12}$ yields the claim.
\end{proof}

Now we are in position to prove Theorem \ref{thm:approximation}. So far we have only shown that the maximum over the local likelihood ratio statistics can be approximated by Gaussian versions, but we did not include the scale penalization $\pen{v}{\left|R\right|}$ in \eqref{eq:defi_pen}. To include this in the approximation result, we will slice the maximum into scales, where the penalty-term is almost constant. Then, we show that we may bound the maximum over all scales by the sum of the maximum over theses families. The price to pay is an additional $\log(n)$ factor on the smallest scale.
\begin{proof}[Proof of Theorem \ref{thm:approximation}]
(a) It follows from the triangle inequality 
\[
\left|\|x\|_\infty -\|y\|_\infty\right| \leq \|x-y\|_\infty,
\]
Lemma \ref{prop.approx.hyperrectangles} and \eqref{eq:r_n} that
\begin{align*}
&\left|\max\limits_{\substack{R \in \mathcal R_n:\\ |R| \geq r_n}} \left(T_R(Y, \theta_0) - \pen{v}{\left|R\right|} \right)- \right.\\
& \hspace{0.7cm}\left.\max\limits_{\substack{R \in \mathcal R_n:\\ |R| \geq r_n}} \left( |R|^{1/2} \left| \frac{\overline{Y}_R - m(\theta_0)}{\sqrt{v(\theta_0)}}\right|-\pen{v}{\left|R\right|}\right)\right| =  \OP\left(\left(\frac{\log^3(n)}{r_n}\right)^{1/4}\right).
\end{align*} 
Define 
\begin{align*}
Y^R&:= |R|^{-1/2} \sum_{i \in R} \left( \frac{Y_i - m(\theta_0)}{\sqrt{v(\theta_0)}}\right)\\
X^R&:= |R|^{-1/2} \sum_{i \in R} X_i, \quad X_i \overset{i.i.d.}{\sim} \mathcal{N}(0,1).
\end{align*}
With this notation and a symmetry argument
we find from the proof of Theorem \ref{lemma:Chernozhukov_NEF} with {$a_i\equiv1$} that
\begin{align*}
\P{}{\left| \max\limits_{\substack{R \in \mathcal R_n:\\ |R| \geq r_n}} \left|Y^R\right| - \max\limits_{\substack{R \in \mathcal R_n:\\ |R| \geq r_n}} \left| X^R\right| \right|>\delta }\lesssim \delta^{-3}\left(\frac{\log^{10}(n)}{r_n}\right)^{1/2}.
\end{align*}
Let $\delta_n:= ((\log^{12}(n)/r_n)^{1/10} \searrow 0.$
Now define $\epsilon_j := j \delta_n, j \in \N$ and 
\[
\mathcal{R}_{n,j} := \left\lbrace R \in \mathcal R_n \left| \right. \exp(\epsilon_j) < |R| < \exp(\epsilon_{j+1}) \right\rbrace.
\]
Then the set of candidate regions $\mathcal R_n$ can be written as 
\[
\mathcal R_{n|r_n} = \bigsqcup\limits_{j \in J} \mathcal{R}_{n,j}, \quad J:= \left\lbrace  \frac{1}{\delta_n}\log \left(\log^{12}(n)\right), \ldots, \frac{1}{\delta_n}\log(n^d)\right\rbrace
\]
with $|J| \leq \frac{\log(n^d)}{\delta_n}$. 
If we abbreviate
\begin{align*}
\p_j := \pen{v}{\exp\left(\epsilon_j\right)} = \sqrt{2v \left(\log \left( \frac{n^d}{\exp(\epsilon_j)}\right) + 1\right)},
\end{align*}
then the slicing above implies 
\begin{align*}
\p_{j+1} \leq \pen{v}{\left|R\right|} \leq \p_j, \qquad\text{for all}\qquad R \in \mathcal{R}_{n,j}. 
\end{align*}
Using $\sqrt{a}- \sqrt{b} = (a-b)/ \left(\sqrt{a}+\sqrt{b}\right)$, we get
\begin{align*}
0 &\leq \p_j - \p_{j+1} \\ 
&= \frac{2v(\epsilon_{j+1}- \epsilon_j)}{\sqrt{2v\left[ \log(n^d) + 1 - \epsilon_j\right]} + \sqrt{2v\left[ \log(n^d) + 1 - \epsilon_{j+1}\right]}}.
\end{align*}
The largest index in $J$ is $\frac{1}{\delta_n}\log(n^d)$ and therefore the maximal value of $\epsilon_i$ is given by $\bar \epsilon= \log(n^d)$ and  $\log(n^d) + 1 - \bar{\epsilon} =1$. Therefore,
\begin{align*}
0 \leq \p_j - \p_{j+1} \leq \frac{2v(\epsilon_{j+1}-\epsilon_j)}{2\sqrt{2v}} = \delta_n \sqrt{\frac{v}{2}}.
\end{align*}
This means that for $n \rightarrow \infty$ the penalty terms $\pen{v}{\left|R\right|}$, $R \in \mathcal{R}_{n,j}$ can be considered as constant. 
{Therefore by straight forward computations, $|J| \leq \frac{\log(n^d)}{\delta_n}$ and choosing $\delta_n \leq \frac{\epsilon}{2}$ we derive
\begin{align*}
&\P{}{\left| \max\limits_{\substack{R \in \mathcal R_n: \\ |R| \geq r_n}} \left(\left| Y^R\right| - \pen{v}{\left|R\right|}\right) - \max\limits_{\substack{R \in \mathcal R_n: \\ |R| \geq r_n}} \left(\left| X^R\right| - \pen{v}{\left|R\right|}\right) \right| \geq \epsilon}\\ 
&\leq \P{}{\max\limits_{j \in J} \left| \max\limits_{R \in \mathcal{R}_{n,j}} \left| Y^R\right| - \max\limits_{R \in \mathcal{R}_{n,j}} \left| X^R\right|\right|  \geq \frac{\epsilon}{2}}\\
&\leq \sum\limits_{j \in J} \P{}{\left| \max\limits_{R \in \mathcal{R}_{n,j}} \left| Y^R\right| - \max\limits_{R \in \mathcal{R}_{n,j}} \left| X^R\right|\right|  \geq \frac{\epsilon}{2}}\\
&\leq |J| \frac{\delta_n^2}{\log(n^d)} = \delta_n \searrow 0, \quad n \rightarrow \infty.
\end{align*}}

\smallskip 

(b) This is a direct consequence of (a).
\end{proof}

We will now continue with the proof of Theorem \ref{thm:weak_limit}. Taking into account the result of Theorem \ref{thm:approximation},we only have to prove {an invariance principle and exploit the continuous mapping theorem}, which will be done in the following. 
%
\begin{lem}\label{prop.invariance}
Let $\mathcal R^*$ satisfy Assumption \ref{ass:R_complex} and be equipped with the canonical metric $\rho^*$ as in \eqref{eq:metric} and define $\mathcal R_n$ as in \eqref{eq:defi_R_n}. Furthermore let $W$ denote white noise on $\left[0,1\right]^d$. For $X_i \stackrel{\text{i.i.d.}}{\sim} \mathcal{N}(0,1)$, $i \in I_n^d$ define
\[
Z_n(R^*) := n^{-d/2}\sum\limits_{i/n \in R^*} X_i \stackrel{\mathcal{D}}{=} n^{-d/2}\sum\limits_{i \in \{1, \ldots, n\}^d} \left|nR^*\cap A_i\right| X_i, \qquad R^* \in \mathcal R^*
\]
where $A_i = (i_1-1,i_1] \times \ldots \times (i_d-1,i_d]$ is the unit cube with upper corner $i$. Then it holds
\begin{align*} 
Z_n \overset{\mathcal{D}}{\rightarrow} W, \quad n \rightarrow \infty.
\end{align*}
\end{lem}
\begin{proof}
Note that $\mathcal{R}^*$ is totally bounded w.r.t. $\rho^*$. We will show the assumptions of  \citep[][Thm. 2.1]{kosorok2007introduction}:
\begin{enumerate}
\item \textit{Tightness:} The white noise $W$ is tight.
\item \textit{Totally boundedness:} 
By Markov's inequality and standard bounds on the modulus of continuity, we obtain using Assumption \ref{ass:R_complex}(a) that
\begin{align*}
&\mathbb{P}^*_{}\left[ \sup\limits_{\substack{R_1^*, R_2^* \in \mathcal{R}^*\\ \rho(R_1^*,R_2^*)\leq  \delta}} |Z_n(R_1^*)-Z_n(R_2^*)|>\epsilon\right]\\
\leq&\frac{1}{\epsilon} \E{}{\sup\limits_{\substack{R_1^*, R_2^* \in \mathcal{R}^*\\ \rho(R_1^*, R_2^*)\leq  \delta}} \left| Z_n(R_1^*)-Z_n(R_2^*)\right| }\\
\lesssim & \int_0^\delta {\sqrt{2\nu \left(\mathcal{R}^*\right)\log \left(\frac{c}{u} \right) }} \,\mathrm du,\\
\end{align*}
which tends to $0$ as $\delta \searrow 0$.
\item 
\textit{Finite dimensional convergence}: The convergence of the finite-dimensional laws is an application of the central limit theorem for random fields \citep[][Thm 2.2]{dedecker1998central} and \citep[][Lemma 2]{dedecker2001exponential}, which shows that
\[
\frac{\left|nR^*\cap\mathbb{Z}^d\right|}{n^d} \rightarrow \left|R^*\right|
\]
for regular Borel sets $R^* \subset [0,1]^d$ with $\left|R^*\right|)>0.$ Consequently, the central limit theorem shows for any fixed $R^* \in \mathcal R^*$ that
\begin{align*}
Z_n(R^*) \stackrel{\mathcal{D}}{\rightarrow} N\left(0, \left|R^*\right| \right) \qquad\text{as}\qquad n \to \infty
\end{align*}
A similar computation shows that
\begin{align*}
\text{Cov}\left( Z_n(R_1^*), Z_n(R_2^*)\right) 
&\to \left|R_1^* \cap R_2^*\right|
\end{align*}
for all $R_1^*, R_2^* \in \mathcal R$. This shows finite dimensional convergence.
\end{enumerate}
\end{proof}

Now we want to apply the generalized version of the continuous mapping theorem \citep[see e.g.][Thm. 5.5]{billingsley2013convergence}. For $c \geq 0$ and $x \in \mathcal{C}(\mathcal{B}([0,1]^d), \R),$ where $\mathcal{B}([0,1]^d)$ denote the Borel sets of $[0,1]^d$ define
\begin{align*}
h^c(x)&:= \sup\limits_{\substack{R^*\in \mathcal{R}^*: \\ |R^*|>c^d}} \left( \frac{|x(R^*)|}{\sqrt{|R^*|}} - \pen{v}{n^d\left|R^*\right|}\right)\\
h_n^c(x)&:= \max\limits_{\substack{R \in \mathcal R_n: \\ |R|> (cn)^d}} \left(\frac{\left|x(R/n)\right|}{\sqrt{|R|/ n^d}} - \pen{v}{\left|R\right|} \right).
\end{align*}
The necessary conditions to apply the continuous mapping theorem are given by the following Lemma:
\begin{lem}\label{lem:cont}
Consider $h^c, h_n^c$ as functions $\left(\mathcal{C}(\mathcal{B}([0,1]^d), \R), \|\cdot\|_\infty\right) \rightarrow \R$. 
\begin{enumerate}
\item[i)] $h^c$ is uniformly continuous and $\left(h_n^c\right)_{n \in \N}$ is a sequence of equi-continuous functions, (uniformly in $n$).
\item[ii)] For $(x_n)_{n} \in \mathcal{C}(\mathcal{B}([0,1]^d),\R),$ s.t. $x_n \rightarrow x$ it holds
\[ h_n^c(x_n) \rightarrow h^c(x), \quad n \rightarrow \infty.\]
\end{enumerate}
\end{lem}
\begin{proof}
\begin{enumerate}
\item[i)]
Let $\epsilon>0,$ choose $\delta=\epsilon c^{d/2}$. Consider two functions $x,y \in \mathcal{C}(\mathcal{B}([0,1]^d), \R)$ s.t. $d(x,y) = \sup\limits_{R^* \subset [0,1]^d} \left||x(R^*)|-|y(R^*)|\right| < \delta.$
By using $\left|\max a_i - \max b_i \right| \leq \max |a_i - b_i|$ we find
\begin{align*}
|h_n^c(x)-h_n^c(y)| \leq \max_{\substack{R \in \mathcal R_n\\  |R|>(cn)^d}} \left| \frac{|x(R/n)|-|y(R/n)| }{\sqrt{|R|/n^d}} \right|\leq \frac{\delta}{c^{d/2}}=\epsilon.
\end{align*} 
Similar arguments yield the uniform continuity of $h^c.$\\
\item[ii)]
Let $(x_n)_n,~ x \in \mathcal{C}(\mathcal{B}([0,1]^d),\R),$ s.t. $x_n \rightarrow x.$ Since the functions $(h_m^c)_{m \in \N}$ are equi-continuous, for any $\epsilon > 0$ we can find an $N_1 \in \N$ s.t. $\forall n > N_1 ~\forall m:$
\begin{align*}
|h_m^c(x_n) - h_m^c(x)| < \frac{\epsilon}{2}.
\end{align*}
Given $\epsilon$ and $N_1$ and $n > N_1$ with $|h_m^c(x_n) - h_m^c(x)| < \frac{\epsilon}{2}$, choose $m=n.$ Then
\begin{align}\label{ContinuityOfh_n^c}
|h_n^c(x_n)-h_n^c(x)| < \epsilon/2.
\end{align}
Now let us define
\[
\mathcal{A}:= \left\lbrace R^* \in \mathcal{R}^*: |R^*| \geq c^d\right\rbrace, \quad \mathcal{B}_n := \left\lbrace R/n  \in \mathcal{R}^* : R \in \mathcal R_n, |R| \geq (cn)^d\right\rbrace.
\]
The set $\mathcal{A}$ is a compact set w.r.t. the metric $\rho^*$ defined in \eqref{eq:metric}, w.r.t. which  $\mathcal{R}^*$ is totally bounded. Furthermore $\mathcal{B}_n$ is a finite subset of $\mathcal{A}.$ If we fix $x \in \mathcal B (\left[0,1\right]^d)$ and introduce $g:\mathcal{A} \rightarrow \R$ by
\begin{align*}
g(R^*):= \left(\frac{|x(R^*)|}{\sqrt{|R^*|}} - \pen{v}{\left|R^*\right|}\right), \qquad R^* \in \mathcal{R}^*,
\end{align*}
then it holds
\begin{align}\label{UpperBoundOnh_n}
h^c(x) = \sup_{R^* \in \mathcal{A}} g(R^*) {\geq} h_n^c(x) = \max_{R^* \in \mathcal{B}_n} g(R^*).
\end{align}
since $\mathcal{B}_n$ is a subset of $\mathcal{A}$. Straight forward computations show that $g$ is continuous w.r.t. $\rho^*$, which implies by compactness of $\mathcal A$ that there exists an $\widetilde{R} \in \mathcal A$ s.t. $h^c(x)= g(\widetilde{R}).$
Now let $R_n \in \mathcal{B}_n$ be a sequence s.t. $R_n \rightarrow \widetilde{R}, n \rightarrow \infty$ w.r.t. $\rho$. Then $g(R_n) \rightarrow g(\widetilde{R})$ as $n \rightarrow \infty$ and hence
\begin{align*}
h^c(x) \overset{\eqref{UpperBoundOnh_n}}{\geq} h_n^c(x) \geq g(R_n) \rightarrow g(\tilde{R}) = h^c(x).
\end{align*}
Consequently there exists a $N_2 \in  \N$ s.t. $\forall n > N_2$ it holds
\begin{align*}
|h^c(x)-h_n^c(x)| < \epsilon/2,
\end{align*}
which together with \eqref{ContinuityOfh_n^c} implies
\begin{align*}
|h_n^c(x_n)- h^c(x)| \leq \epsilon \qquad\text{for all}\qquad n > \max\{N_1,N_2\}.
\end{align*}
\end{enumerate}
\end{proof}

Now we are in position to prove Theorem \ref{thm:weak_limit}:
\begin{proof}[Proof of Theorem \ref{thm:weak_limit}]
By Lemma \ref{prop.invariance}, Lemma \ref{lem:cont} and the generalized version of the continuous mapping theorem \citep[see e.g.][Thm. 5.5]{billingsley2013convergence} we get
\begin{align*}
h_n^c(Z_n) \overset{\mathcal{D}}{\rightarrow} h^c(W), \quad n \rightarrow \infty.
\end{align*}
The functions $h_n^c$ and $h^c$ have been defined such that
\begin{align*}
h_n^c(Z_n)= M_n\left(\mathcal R_{n |  (cn)^d},v \right),\quad \text{and}\quad
h^c(W)=M\left(\mathcal{R}^*_{| c^d},v\right),
\end{align*}
i.e. for all $c>0$ holds
\begin{align*}
M_n\left(\mathcal R_{n \left|  \right. (cn)^d },v\right) \overset{\mathcal{D}}{\rightarrow} M\left(\mathcal{R}^*_{\left| \right. c^d},v \right) \qquad\text{as}\qquad n \to \infty.
\end{align*}
Since $M\left(\mathcal{R}^*_{\left| \right. c^d} ,v\right)\overset{\mathcal{D}}{\rightarrow}  M\left(\mathcal{R}^*,v\right), ~ c \rightarrow 0$, we get
\[
\lim\limits_{c \rightarrow 0} \lim\limits_{n \rightarrow \infty} M_n\left(\mathcal R_{n \left|  \right. (cn)^d},v \right) = M\left(\mathcal{R}^*,v\right).
\]
It can also be readily seen from the definition of $M_n$ and $M$ that
\[
\liminf_{n \rightarrow \infty} \P{}{M_n\left(\mathcal R_{n \left| \right. r_n },v\right) \leq t} \geq \P{}{M\left(\mathcal{R}^*,v \right) \leq t}.
\]
Now let $c>0$ be fixed and assume $r_n < (cn)^d$ for all $n \in \N$. Then we obtain altogether that
\begin{align*}
\P{}{M\left(\mathcal{R}^*,v \right) \leq t} & \leq \liminf_{n \rightarrow \infty} \P{}{M_n\left(\mathcal R_{n \left| \right. r_n},v\right) \leq t} \\
& \leq \limsup_{n \rightarrow \infty} \P{}{M_n\left(\mathcal R_{n \left| \right. (cn)^d},v \right) \leq t} \\
& \to \P{}{M\left(\mathcal{R}^* ,v\right) \leq t} \qquad\text{as}\qquad c \searrow 0,
\end{align*}
which yields 
\[
M_n \left(\mathcal R_{n|r_n}, v\right) \toD M\left(\mathcal R^*,v\right) \qquad\text{as}\qquad n \to \infty.
\]
This proves the main statement. It remains to show a.s. boundedness and non-degenerateness of $M \left(\mathcal R^*, v\right)$. We apply \citep[][Thm. 6.1]{ds01} with $\rho^*$ as in \eqref{eq:metric} and 
\begin{align*}
\sigma^2(R^*):=\left|R^*\right|,\qquad X(R^*) :=W(R^*).
\end{align*}
Let us check the three conditions from their theorem:\\
\textit{i)} $\sigma^2(R_1^*) \leq \sigma^2(R_2^*) +  \rho^*(R_1^*, R_2^*)^2$ for all $R_1^*, R_2^* \in \mathcal{R}^*$ is obviously fulfilled since $R_1^* \cap R_2^* \subset R_2^*$ and $R_1^* \setminus R_2^* \subset R_1^* \bigtriangleup R_2^*$. Since $\V{W(R^*)}= \left|R^*\right|$, 
\begin{align}\label{eq.DSpokTheorem_i}
\P{}{X(R^*) > \sigma(R^*)\eta} = \P{}{W(R^*)  > \eta ~(\left|R^*\right|)^{1/2}}\leq \frac{1}{2} \exp\left(-\frac{\eta^2}{2}\right).
\end{align}
\textit{ii)} For
\[
\P{}{|X(R_1^*)-X(R_2^*)|>\rho(R_1^*, R_2^*)\eta}  = \P{}{ \left|  W(R_1^*) - W(R_2^*)\right| > \left|R_1^* \bigtriangleup R_2^*\right|)^{1/2} \eta }
\]
we compute that $W(R_1^*)-W(R_2^*) \sim \mathcal{N}(0, \sigma^2_{R_1^*, R_2^*})$, $\sigma^2_{R_1^*, R_2^*}= \left|R_1^*\right| + \left|R_2^*\right| - 2 \Cov(W(R_1^*), W(R_2^*))$ and $\left|R_1^* \bigtriangleup R_2^*\right|= \left|R_1^*\right| + \left|R_2^*\right| - 2 \left|R_1^* \cap R_2^*\right|$. With $\Cov(W(R_1^*), W(R_2^*)) = \left|R_1^* \cap R_2^*\right|$ we consequently find
\[
\P{}{|X(R_1^*)-X(R_2^*)|>\rho(R_1^*, R_2^*)\eta} \leq  \exp\left(-\frac{\eta^2}{2} \frac{\left|R_1^* \bigtriangleup R_2^*\right|}{\sigma^2_{R_1^*,R_2^*}}\right) = \exp\left(-\frac{\eta^2}{2}\right).
\]
\textit{iii)} Is fulfilled by Assumption \ref{ass:v} (cf. \eqref{eq:capacity_bound}).

\smallskip

\eqref{eq.DSpokTheorem_i} holds with $- W(R^*)$ as well, hence we get that the statistic $M\left(\mathcal{R}^* , v\right) < \infty$ a.s.
Non-degenerateness is obvious, as $M$ is always larger than the value of the local statistic on one fixed scale, which is non-degenerate. 
\end{proof}

\subsection{Proofs of Section \ref{subsec:power}}
Let us now prove the results from Section \ref{subsec:power}, namely Theorem \ref{thm:asymptPower} and Corollary \ref{cor:power}. First we introduce some abbreviations to ease notation. Let
\[
q^* := q_{1-\alpha,n}^{\mathrm O}, \qquad q := q_{1-\alpha,n}^{\mathrm{MS}}
\]
and denote the total signal on $Q \in \mathcal Q^n$ by
\begin{equation}\label{eq:signal}
\mu^n \left(Q\right) := \left|Q\right|^{-1/2} \sum\limits_{i \in Q}  \frac{m(\theta_i^n)-m(\theta_0)}{\sqrt{v(\theta_0)}} = \frac{\left|Q\cap Q_n\right|}{\sqrt{\left|Q\right|}} \frac{m\left(\theta_1^n\right) - m\left(\theta_0\right)}{\sqrt{v\left(\theta_0\right)}}.
\end{equation}
For brevity introduce the Gaussian process 
\[
\gamma \left(Q\right) : = \left| \mu^n \left(Q\right) + |Q|^{-\frac{1}{2}}\sum_{i \in Q}v_i X_i\right| - \pen{v}{\left|Q\right|}, \qquad Q \in \mathcal Q^n
\]
with $X_i \stackrel{\text{i.i.d.}}{\sim} \mathcal N \left(0,1\right)$ and $v_i = \sqrt{v\left(\theta_i\right) / v\left(\theta_0\right)}$. 

Let us now start with the analysis of the oracle procedure. As a preparation we require to leave out a suitable subset of hypercubes close to the true anomaly $Q_n$. Therefore, choose a sequence $\varepsilon_n$ such that $\varepsilon_n \searrow 0$ but $\varepsilon_n \mu^n \left(Q_n\right) \to \infty$ and denote the set of all hypercubes which are close to the anomaly by
\[
\mathcal U_n := \left\{Q \in \mathcal Q^n\left(a_n\right)  ~\big|~\mu^n \left(Q\right) \geq \mu^n\left(Q_n\right) \left(1-\varepsilon_n\right) \right\}.
\]
Furthermore define the extended neighborhood of the anomaly by
\[
\mathcal U := \left\{Q \in \mathcal Q^n\left(a_n\right) ~ \big|~ Q \cap Q' \neq \emptyset \text{ for some } Q' \in \mathcal U_n\right\},
\]
its complement by $\mathcal T := \mathcal Q^n \left(a_n\right)\setminus \mathcal U$. By definition, $\left\{\gamma\left(Q\right)\right\}_{Q \in \mathcal T}$ and $\left\{\gamma\left(Q\right)\right\}_{Q \in \mathcal U_n}$ are independent, which will allow us to compute the asymptotic power of the single-scale procedure. For a sketch of $\mathcal U_n$ and $\mathcal U$ see Figure \ref{fig:oracle}.

\begin{figure}[!htb]
\centering
\begin{tikzpicture}[scale = .05]

\coordinate (A) at (0,0);
\coordinate (B) at (100,100);

\coordinate (C) at (40,40);
\coordinate (D) at (50,50);

\coordinate (E) at (49,49);
\coordinate (F) at (59,59);

\coordinate (G) at (33,37);
\coordinate (H) at (43,47);

\coordinate (I) at (57,55);
\coordinate (J) at (67,65);

\coordinate (Q) at (57,40);
\coordinate (R) at (67,50);

\coordinate (K) at (26,30);
\coordinate (L) at (36,40);

\coordinate (M) at (5,5);
\coordinate (N) at (15,15);

\coordinate (O) at (77,77);
\coordinate (P) at (87,87);

\coordinate (S) at (65,12);
\coordinate (T) at (75,22);

\coordinate (U) at (15,80);
\coordinate (V) at (25,90);

\draw (A) rectangle (B);
\draw[fill=red] (C) rectangle (D);

\draw[pattern = north east lines] (E) rectangle (F);
\draw[pattern = north east lines] (G) rectangle (H);

\draw[pattern = crosshatch dots] (I) rectangle (J);
\draw[pattern = crosshatch dots] (K) rectangle (L);
\draw[pattern = crosshatch dots] (Q) rectangle (R);

\draw[fill=black] (M) rectangle (N);
\draw[fill=black] (O) rectangle (P);
\draw[fill=black] (S) rectangle (T);
\draw[fill=black] (U) rectangle (V);

\end{tikzpicture} 
\caption{Exemplary elements of the sets $\mathcal U_n, \mathcal U$ and $\mathcal T$ in $d = 2$: The anomaly is shown in red, the hatched cubes belong to $\mathcal U_n$, the dotted cubes to $\mathcal U$, and all black cubes belong to $\mathcal T$. By definition, for all $Q \in \mathcal U_n$ and $Q' \in \mathcal T$ it holds $Q \cap Q' = \emptyset$, which implies independence of $\left\{\gamma\left(Q\right)\right\}_{Q \in \mathcal T}$ and $\left\{\gamma\left(Q\right)\right\}_{Q \in \mathcal U_n}$.}
\label{fig:oracle}
\end{figure}
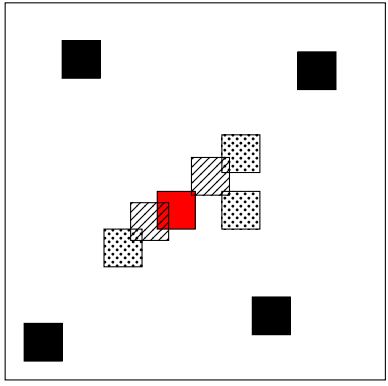

{
We start by bounding the covering number $N(\mathcal{U}, \rho, \epsilon)$ with respect to the canonical metric $\rho\left(Q,Q'\right)^2 = 2-2\left|Q\cap Q'\right|/\sqrt{\left|Q\right|\left|Q'\right|}$.
\begin{lem}\label{lem:coveringnumberU}
For any $\epsilon>0$ we have 
\begin{align*}
N(\mathcal{U}, \rho, \epsilon) \leq C \left(\frac{6d}{\epsilon} \right)^d.
\end{align*}
\end{lem}

\begin{proof}
Let $Q(3)$ denote the the cube of side length $3$ times the side length of $Q_n$ centered at the midpoint of $Q_n$. Let $0<\delta\leq \frac{\epsilon}{2d},~\frac{1}{\delta} \in \N.$ Choose equidistant points in $Q(3)$ of distance $\delta|Q_n|^{1/n}$ in each coordinate, which requires $\left(  \frac{3}{\delta}\right)^d$ points. As a covering $M$ for $\mathcal{U}$, consider the cubes of side length $|Q_n|^{1/n}$ which have vertices in the net of equidistant points. To approximate $Q \in \mathcal{U}$ with $Q \cap Q' \neq \emptyset,$ where $Q' \in \mathcal{U}_n,$ i.e., $\left| Q_n \cap Q'\right| \geq \left( 1-\frac{\delta}{2}\right)|Q_n|,$ by elements of this net, note that $Q$ is essentially contained in $Q(3)$ -up to distance $\frac{\delta}{2}$- and therefore there is a cube $\widetilde{Q} \in M$ in the covering such that the volume (or number of points) in $Q \bigtriangleup \widetilde{Q}$ is bounded by $\leq d \left(\delta |Q_n|^{1/d} \right) |Q_n|^{(d-1)/d},$ since the complements $Q \backslash \widetilde{Q} $ and $ \widetilde{Q} \backslash Q$ in each fixed dimension have at most width $\leq \delta |Q_n|^{1/d}$ and extension $|Q_n|^{1/d}$ in the remaining $(n-1)$ dimensions. It is bounded by $d\delta |Q_n|,~\left|Q\bigtriangleup \widetilde{Q}\right| \geq \left(1-d\delta \right)|Q_n| \geq \left(1-\frac{\epsilon}{2} \right)|Q_n|.$ Therefore $\mathcal{N}(\mathcal{U}, \rho, \epsilon) \leq \left( \frac{3}{\delta}\right)^d \leq C \left(\frac{6d}{\epsilon} \right)^d.$
\end{proof}
}

\begin{lem}\label{lem:distribution_unchanged_oracle}
Consider the setting from Section \ref{subsec:power} and recall that $q^*$ is the $(1-\alpha)-$quantile of $M_n\left(\mathcal Q^n\left(a_n\right)\right)$ as in \eqref{eq:Gaussian_approximation}. Then
\begin{itemize}
\item[(a)] $\max\limits_{Q \in \mathcal U} |Q|^{-\frac{1}{2}}\left|\sum_{i \in Q}v_i X_i \right| = \OP\left(1\right)$ as $n \to \infty$
\item[(b)] $\lim\limits_{n \to \infty} \P{}{\max\limits_{Q \in \mathcal T}\gamma \left(Q\right) \leq q^*} = 1-\alpha$
\end{itemize}
\end{lem}
\begin{proof}

\smallskip

(a) It follows from Dudley's entropy integral \citep[see e.g.][Thm. 6.1.2]{marcus2006markov} with any fixed $Q' \in \mathcal U$ that
\begin{align*}
&\E{}{ \max\limits_{Q \in \mathcal{U}} |Q|^{-1/2} \left| \sum\limits_{i \in Q} v_i X_i\right| } \\
\leq &\E{}{|Q'|^{-1/2} \left| \sum\limits_{i \in Q'}  v_i X_i\right|} +  \E{}{\max\limits_{Q,Q' \in \mathcal{U}} \left||Q|^{-1/2} \left| \sum\limits_{i \in Q} v_i X_i\right|-|Q'|^{-1/2} \left| \sum\limits_{i \in Q'} v_i X_i\right| \right|}\\
\leq &\sqrt{\frac{2\overline{v}^2}{\pi}}+ C \int_0^{2} \sqrt{\log \mathcal{N}(\mathcal{U}, \rho, \epsilon)} \,\mathrm d\epsilon  \leq \sqrt{\frac{2\overline{v}^2}{\pi}}+ C \int_0^{2} {\sqrt{d} \sqrt{\log \left(\frac{6d}{\epsilon} \right)}} \,\mathrm d\epsilon < \infty
\end{align*}
which by Markov's inequality proves the claim.

\smallskip

(b) A direct consequence of (a) is that
\[
\max\limits_{Q \in \mathcal U} \left[|Q|^{-\frac{1}{2}}\left|\sum_{i \in Q}v_i X_i \right| - \pen{v}{\left|Q\right|}\right] = \oP(1).
\]
Furthermore note that $\mu^n(Q)=0$ and $v_i \equiv 1$, $i \in Q$ for $Q \in \mathcal{T}$. Consequently
\begin{align*}
\P{}{\max\limits_{Q \in \mathcal{T}} \gamma \left(Q\right) \leq q^*} 
=&  \P{}{\max\limits_{Q \in \mathcal{T}}\left[|Q|^{-\frac{1}{2}} \left| \sum\limits_{i \in Q} X_i\right|-\pen{v}{\left|Q\right|}\right] \leq q^* }\\
=& \P{}{\max\limits_{Q \in \mathcal Q^n(a_n)} \left[|Q|^{-\frac{1}{2}} \left| \sum\limits_{i \in Q} X_i\right|-\pen{v}{\left|Q\right|}\right] \leq q^* }+ o(1)\\
=& \P{}{M_n(\mathcal Q^n(a_n)) \leq q^*}+ o(1)
\end{align*}
which yields the claim.
\end{proof}
With this Lemma at hand, we are now in position to derive the asymptotic power of the oracle procedure:

\begin{proof}[Proof of Theorem \ref{thm:asymptPower}(a)]
To analyze $\P{\theta^n}{ T_n (Y, \theta_0, \mathcal Q^n(a_n)) > q^*}$, we start with showing a $\geq$ in the statement of Theorem \ref{thm:asymptPower}(a). By Lemma \ref{prop.approx.hyperrectangles} and the triangle inequality we can replace $T_n (Y, \theta_0, \mathcal Q^n(a_n))$ by 
\[
\max_{Q\in\mathcal Q^n\left(a_n\right)} \left[|Q|^{-\frac{1}{2}} \left|\sum_{i \in Q}\frac{Y_i-m(\theta_0)}{\sqrt{v(\theta_0)}}\right| - \pen{v}{\left|Q\right|}\right]
\]
up to $\oP(1)$. 
Furthermore Theorem \ref{lemma:Chernozhukov_NEF} allows us to approximate the latter sum by a Gaussian version, i.e.
\begin{align*}
\P{\theta^n}{ T_n (Y, \theta_0, \mathcal Q^n(a_n)) > q^*} = \P{}{\max_{Q\in\mathcal Q^n\left(a_n\right)}\gamma\left(Q\right)  > q^*} + o(1).
\end{align*}
Now we derive
\begin{align*}
&\P{}{\max_{Q\in\mathcal Q^n\left(a_n\right)}\gamma \left(Q\right) > q^*} \\
= &\P{}{\left\{\max_{Q\in\mathcal Q^n\left(a_n\right)}\gamma \left(Q\right) > q^*\right\} \cap \left\{\max_{Q\in\mathcal T}\gamma \left(Q\right) \leq q^* \right\}} \\
& + \P{}{\left\{\max_{Q\in\mathcal Q^n\left(a_n\right)}\gamma \left(Q\right) > q^*\right\} \cap \left\{\max_{Q\in\mathcal T}\gamma \left(Q\right) > q^* \right\}}\\
= &\P{}{\left\{\max_{Q \in \mathcal T}\gamma \left(Q\right) \leq q^*\right\} \cap \left\{\max_{Q\in\mathcal U}\gamma \left(Q\right) > q^* \right\}} + \P{}{\max_{Q\in\mathcal T}\gamma \left(Q\right) > q^*}\\
\geq& \P{}{\left\{\max_{Q \in \mathcal T}\gamma \left(Q\right) \leq q^*\right\} \cap \left\{\gamma \left(Q_n\right) > q^* \right\}} + \P{}{\max_{Q\in\mathcal T}\gamma \left(Q\right) > q^*}\\
= &\P{}{\max_{Q \in \mathcal T}\gamma \left(Q\right) \leq q^*} \P{}{\gamma \left(Q_n\right) > q^*} + \P{}{\max_{Q\in\mathcal T}\gamma \left(Q\right) > q^*}
\end{align*}
where we exploited $Q_n \in \mathcal U$ and independence of $\left\{\gamma\left(Q\right)\right\}_{Q \in \mathcal T}$ and $\gamma \left(Q_n\right)$. Lemma \ref{lem:distribution_unchanged_oracle}(b) states that $\P{}{\max_{Q \in \mathcal T}\gamma \left(Q\right) \leq q^*} = 1-\alpha + o(1)$ and hence
\[
\P{\theta^n}{ T_n (Y, \theta_0, \mathcal Q^n(a_n)) > q^*} \geq \alpha + \left(1-\alpha\right) \P{}{\gamma \left(Q_n\right) > q^*}+ o(1).
\]
Furthermore note that $\gamma \left(Q_n\right) + \pen{v}{\left|Q_n\right|}$ follows a folded normal distribution with parameters $\mu = \mu^n\left(Q_n\right)$ and $\sigma^2 = \left|Q_n\right|^{-1} \sum_{i \in Q_n} v_i^2$, this is
\[
\gamma \left(Q_n\right)\sim \left|\mathcal N \left(\mu, \sigma^2\right)\right| - \pen{v}{\left|Q_n\right|}.
\]
We compute
\begin{align}
\mu^n\left(Q_n\right) &= \sqrt{n^d a_n}   2,5 C\frac{m\left(\theta_1^n\right) - m\left(\theta_0\right)}{\sqrt{v\left(\theta_0\right)}} \left(1+o\left(1\right)\right), \label{eq:munSn}\\
\left|Q_n\right|^{-1} \sum_{i \in Q_n} v_i^2 &=  \frac{v\left(\theta_1^n\right)}{v\left(\theta_0\right)}, \label{eq:varianceSn}\\
\pen{v}{\left|Q\right|} & = \sqrt{2v \log\left(a_n^{-1}\right)} + o(1) \quad\text{for all}\quad Q \in \mathcal Q^n \left(a_n\right), \label{eq:pen}
\end{align}
which yields by continuity of $F$ and $Q_n \in \mathcal Q^n\left(a_n\right)$ the proposed lower bound. For the upper bound (i.e. $\leq$ in the statement of Theorem \ref{thm:asymptPower}(a)) we proceed as before and obtain
\begin{align*}
&\P{\theta^n}{ T_n (Y, \theta_0, \mathcal Q^n(a_n)) > q^*}\\
= &\alpha + \P{}{\left\{\max_{Q \in \mathcal T}\gamma \left(Q\right) \leq {q^*}\right\} \cap \left\{\max_{Q\in\mathcal U}\gamma \left(Q\right) > q^* \right\}} + o(1)\\
= &\alpha + \P{}{\left\{\max_{Q\in \mathcal T}\gamma \left(Q\right) \leq q^*\right\} \cap \left\{\max_{Q\in\mathcal U}\gamma \left(Q\right) > q^* \right\}\cap \left\{\max\limits_{Q \in \mathcal U_n} \gamma \left(Q\right) > q^* \right\}} \\
& +  \P{}{\left\{\max_{Q \in \mathcal T}\gamma \left(Q\right) \leq q^*\right\} \cap \left\{\max_{Q\in\mathcal U}\gamma \left(Q\right) > q^* \right\}\cap \left\{\max\limits_{Q \in \mathcal U_n} \gamma \left(Q\right) \leq q^* \right\}} + o(1)\\
\leq& \alpha + \P{}{\left\{\max_{Q \in \mathcal T}\gamma \left(Q\right) \leq q^*\right\} \cap \left\{\max\limits_{Q \in \mathcal U_n} \gamma \left(Q\right) > q^* \right\}} +  \P{}{\max_{Q\in \mathcal U \setminus \mathcal U_n}\gamma \left(Q\right) > q^*} + o(1)\\
=& \alpha + \left(1-\alpha\right) \P{}{\max\limits_{Q \in \mathcal U_n} \gamma \left(Q\right) > q^*} +  \P{}{\max_{Q \in \mathcal U \setminus \mathcal U_n}\gamma \left(Q\right) > q^*} + o(1)
\end{align*}
where we used independence of $\left\{\gamma \left(Q\right)\right\}_{Q \in \mathcal T}$ and $\left\{\gamma \left(Q\right)\right\}_{Q \in \mathcal U_n}$. From Lemma \ref{lem:distribution_unchanged_oracle}(a) we obtain
\[
\max_{Q \in \mathcal U \setminus \mathcal U_n} \left||Q|^{-\frac{1}{2}}\sum_{i \in Q}v_i X_i \right| \leq \max_{Q \in \mathcal U} \left||Q|^{-\frac{1}{2}}\sum_{i \in Q}v_i X_i \right|  = \OP\left(1\right)
\]
and further by definition of $\mathcal U_n$ that $\mu^n \left(Q\right) \leq \left(1-\varepsilon_n\right) \mu^n \left(Q_n\right)$ for all $Q \in \mathcal U \setminus \mathcal U_n$. Exploiting \eqref{eq:pen} this implies
\begin{align*}
&\P{}{\max_{Q \in \mathcal U \setminus \mathcal U_n}\gamma \left(Q\right) > q^*} \\
\leq &\P{}{\left(1-\varepsilon_n\right) \mu^n \left(Q_n\right) + \OP(1) - \sqrt{2 \log\left(a_n^{-1}\right)} > q^*}\\
= &\P{}{\mu^n \left(Q_n\right) - \sqrt{2 \log\left(a_n^{-1}\right)} - \varepsilon_n \mu^n\left(Q_n\right) + \OP(1)> q^*} = o(1)
\end{align*}
if $\mu^n\left(Q_n\right) - \sqrt{2 \log\left(a_n^{-1}\right)} \to C \in \left[-\infty, \infty\right)$ (as $\varepsilon_n \mu^n\left(Q_n\right) \to \infty$ by construction), and if $\mu^n\left(Q_n\right) - \sqrt{2 \log\left(a_n^{-1}\right)} \to \infty$, {the power trivially converges to $1$}. Altogether this gives
\[
\P{\theta^n}{ T_n (Y, \theta_0, \mathcal Q^n(a_n)) > q^*} \leq \alpha + \left(1-\alpha\right) \P{}{\max\limits_{Q \in \mathcal U_n} \gamma \left(Q\right) > q^*}  + o(1).
\]
With similar arguments as in Lemma \ref{lem:distribution_unchanged_oracle} we obtain from $\varepsilon_n \searrow 0$ that
\[
\P{}{\max\limits_{Q \in \mathcal U_n} \gamma \left(Q\right) > q^*} = \P{}{\gamma\left(Q_n\right) + \oP(1)  > q^*}
\]
and hence the claim is proven.
\end{proof}

Now we turn to the multiscale procedure. As here different scales are considered, the set $\mathcal U$ is not large enough any more. {Specifically}, we cannot construct a subset $\mathcal V$ such that $\left\{\gamma\left(Q\right)\right\}_{\mathcal V^c}$ and $\gamma\left(Q_n\right)$ are independent and $\max_{Q \in \mathcal V}\gamma \left(Q\right)$ is still negligible. Due to this, the corresponding proof in \citet{sac16} is incomplete. To overcome this difficulty, we follow the idea to distinguish if the anomaly $Q_n$ has asymptotically an effect on $\gamma \left(Q\right)$ or not. Whenever $Q$ is sufficiently large compared to $Q_n$, the impact will asymptotically be negligible. 

For some sequence $\epsilon_n \searrow 0$ with $\epsilon_n = O \left(\left|Q_n\right|^{-\gamma}\right)$ with some $\gamma>0$ we introduce
\begin{align}
\delta_n &:= \epsilon_n \max\left\{\mu^n \left(Q_n\right), \log(n)\sqrt{\frac{\left|Q_n\right|}{r_n}} \right\}^{-1}, \label{eq:deltan}\\
\mathcal V &:= \left\{ Q \in {\mathcal Q_{n| r_n}^{\mathrm{MS}} } ~\big|~\mu^n \left(Q\right) \geq \delta_n \mu^n\left(Q_n\right)\right\} \nonumber
\end{align}
and its complement $\mathcal T' := {\mathcal Q_{n| r_n}^{\mathrm{MS}} }\setminus \mathcal V$. {The precise definition of $\delta_n$ is the result of terms which have to vanish in the following Lemma \ref{lem:distribution_unchanged_MS}}. For a sketch see Figure \ref{fig:multiscale}. 

\begin{figure}[!htb]
\centering
\begin{tikzpicture}[scale = .05]

\coordinate (A) at (0,0);
\coordinate (B) at (100,100);

\coordinate (C) at (40,40);
\coordinate (D) at (50,50);

\coordinate (E) at (37,37);
\coordinate (F) at (51,51);

\coordinate (G) at (43,43);
\coordinate (H) at (53,53);

\coordinate (I) at (6,6);
\coordinate (J) at (43,43);

\coordinate (K) at (49,31);
\coordinate (L) at (59,41);

\coordinate (M) at (10,90);
\coordinate (N) at (44,56);

\coordinate (O) at (40,40);
\coordinate (P) at (43,43);

\coordinate (Q) at (49,40);
\coordinate (R) at (50,41);

\draw (A) rectangle (B);
\draw[fill=red] (C) rectangle (D);

\draw[pattern = north east lines] (E) rectangle (F);
\draw[pattern = north east lines] (G) rectangle (H);

\draw[pattern = crosshatch dots] (I) rectangle (J);
\draw[pattern = crosshatch dots] (K) rectangle (L);
\draw[pattern = crosshatch dots] (M) rectangle (N);

\draw[fill=black] (O) rectangle (P);
\draw[fill=black] (Q) rectangle (R);

\end{tikzpicture} 
\caption{Exemplary elements of the sets $\mathcal V$ and $\mathcal T'$ in $d = 2$: The anomaly is shown in red, the hatched cubes belong to $\mathcal V$ and the dotted cubes to $\mathcal T'$. However, the intersections marked in black are small enough such that they have asymptotically no influence on $\gamma\left(Q\right)$.}
\label{fig:multiscale}
\end{figure}
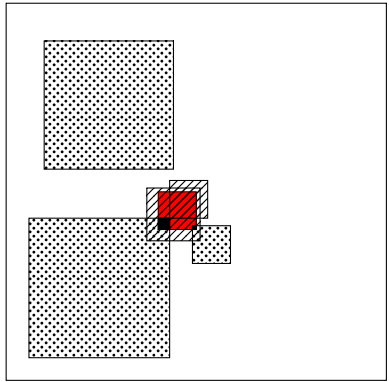

{
Contrary to the oracle procedure, we do not have independence of $\left\{\gamma\left(Q\right)\right\}_{Q \in \mathcal T'}$ and $\gamma\left(Q_n\right)$. However, asymptotically a similar property is true as shown in the following Lemma \ref{lem:distribution_unchanged_MS}. \\
Let us again start with bounding the covering number $N(\mathcal{V}, \rho, \epsilon)$ w.r.t.\ the canonical metric $\rho\left(Q,Q'\right)^2 = 2-2\left|Q\cap Q'\right|/\sqrt{\left|Q\right|\left|Q'\right|}$.

\begin{lem}\label{lem:coveringnumberV}
There exists a constant $C$ such that for any $\epsilon>0$ we have
\begin{align*}
N(\mathcal{V}, \rho, \epsilon) \leq C \left( \frac{6d}{\epsilon}\right)^d \frac{\left|Q_n\right|^{d+1}}{\delta_n^{2 (d+1)}}.
\end{align*} 
\end{lem}
\begin{proof}
For all $Q \in \mathcal V$ defined in \eqref{eq:deltan} it holds $\mu^n \left(Q\right) \geq \delta_n \mu^n\left(Q_n\right)$, which implies
\[
\delta_n \sqrt{\left|Q_n\right|} \leq \frac{\left|Q \cap Q_n\right|}{\sqrt{\left|Q\right|}} \leq \frac{\left|Q_n\right|}{\sqrt{\left|Q\right|}}.
\]
Consequently, $\mathcal V$ contains only cubes $Q$ with $r_n \leq \left|Q\right| \leq \delta_n^{-2} \left|Q_n\right|$. For a fixed scale $k$, $\mathcal V$ contains at most $(C \left|Q_n\right|)$ $Q$'s with $\left|Q\right| = k$, and for the set of such $Q$'s an $\sqrt{\epsilon}$-covering can be constructed as in the proof of Lemma \ref{lem:coveringnumberU} with at most $C \left( \frac{6d}{\epsilon}k\right)^d$
elements, which gives
\[
N(\mathcal{V}, \rho, \epsilon) \leq C\sum\limits_{k=r_n}^{\lfloor \delta_n^{-2} \left|Q_n\right|\rfloor} \left( \frac{6d}{\epsilon}k\right)^d \leq C \left( \frac{6d}{\epsilon}\right)^d \frac{\left|Q_n\right|^{d+1}}{\delta_n^{2 (d+1)}}.
\]
\end{proof}
}

\begin{lem}\label{lem:distribution_unchanged_MS}
Consider the setting from Section \ref{subsec:power} and recall that $q$ is the $(1-\alpha)-$quantile of $M_n \left( {\mathcal Q_{n| r_n}^{\mathrm{MS}} }\right)$ as in \eqref{eq:Gaussian_approximation}. Then the following statements hold true as $n \to \infty$:
\begin{itemize}
\item[(a)] $\max\limits_{Q \in \mathcal V} \left||Q|^{-\frac{1}{2}}\sum_{i \in Q}v_i X_i \right| = \OP\left(\sqrt{\ln\left(\left|Q_n\right|\right)} + \sqrt{\ln\left(-\ln\left(m\left(\theta_1^n\right)-m\left(\theta_0\right)\right)\right)}\right)$
\item[(b)] $\max\limits_{Q \in \mathcal T'} \left| |Q|^{-\frac{1}{2}} \sum\limits_{i \in Q \cap Q_n} v_i X_i\right|  = \oP(1)$
\item[(c)] $\P{}{\max\limits_{Q \in \mathcal T'}\gamma \left(Q\right) \leq q} = 1-\alpha + o(1)$
\end{itemize}
\end{lem}
\begin{proof}

\smallskip

(a) Again with the help of Dudley's entropy integral we find
{\begin{align*}
&\E{}{ \max\limits_{Q \in \mathcal{V}} |Q|^{-1/2} \left| \sum\limits_{i \in Q} v_i X_i\right| } \\
\leq &\sqrt{\frac{2\overline{v}^2}{\pi}}+ C_1 \int_0^{2} \sqrt{\log \mathcal{N}(\mathcal{V}, \rho, \epsilon)} \,\mathrm d\epsilon \\
\leq & C_2 \left( \sqrt{ \ln \left( \frac{\left|Q_n\right|}{\delta_n^2}\right) }\right)\\
\leq & C_3 \left(\sqrt{\ln\left(\left|Q_n\right|\right)} + \sqrt{\ln\left(\left|m\left(\theta_1^n\right)-m\left(\theta_0\right)\right|\right)}\right). 
\end{align*}}
Now Markov's inequality gives the claim.

\smallskip

(b) For $Q \in \mathcal T'$ it holds $\mu^n \left(Q\right) < \delta_n \mu^n \left(Q^n\right)$ and hence $\left|Q \cap Q_n\right| \leq \delta_n \sqrt{\left|Q\right|\left|Q_n\right|}$. Consequently
\begin{align*}
&\E{}{ \max\limits_{Q \in \mathcal{T}'} |Q|^{-1/2} \left| \sum\limits_{i \in Q \cap Q_n} v_i X_i\right| }\\
&{=\E{}{ \max\limits_{Q \in \mathcal{T}'} \sqrt{\frac{|Q\cap Q_n|}{|Q|}} \frac{1}{\sqrt{|Q\cap Q_n|}} \left| \sum\limits_{i \in Q \cap Q_n} v_i X_i\right| }}\\
&{\leq\E{}{ \max\limits_{Q \in \mathcal{T}'} \sqrt{\frac{ \delta_n \sqrt{\left|Q_n\right|}}{\sqrt{|Q|}}} \frac{1}{\sqrt{|Q\cap Q_n|}} \left| \sum\limits_{i \in Q \cap Q_n} v_i X_i\right| }}\\
\leq & \sqrt{\delta_n} \left(\frac{\left|Q_n\right|}{r_n}\right)^{\frac14} \E{}{ \max\limits_{Q \in \mathcal Q_n|r_n} |Q\cap Q_n|^{-1/2} \left| \sum\limits_{i \in Q \cap Q_n} v_i X_i\right| }\\
\leq & C\bar v\sqrt{\delta_n} \left(\frac{\left|Q_n\right|}{r_n}\right)^{\frac14} \sqrt{\log(n)}
\end{align*}
where we used \eqref{eq:Gaussian_partial_sums_max}. As the right-hand side converges to $0$ by \eqref{eq:deltan}, this proves the claim.

\smallskip

(c) This can now be deduced from (a) and (b) as follows. For all $Q \in \mathcal T'$ it holds $\mu^n \left(Q\right) \leq \delta_n \mu^n \left(Q_n\right)$ and hence
\begin{align}
&\max\limits_{Q \in \mathcal T'} \gamma\left(Q\right) - \max\limits_{Q \in \mathcal T'} \left[ \left| \left|Q\right|^{-\frac12} \sum\limits_{i \in Q \setminus Q_n} v_i X_i \right| - \pen{v}{\left|Q\right|}\right] \nonumber\\
\leq & \max\limits_{Q \in \mathcal T'} \left[\mu^n \left(Q\right) + \left| \left|Q\right|^{-\frac12} \sum\limits_{i \in Q \cap Q_n} v_i X_i\right| \right] \nonumber\\
\leq & \delta_n \mu^n \left(Q_n\right) + \max\limits_{Q \in \mathcal T'} \left| \left|Q\right|^{-\frac12} \sum\limits_{i \in Q \cap Q_n} v_i X_i\right|  \stackrel{\mathrm{(b)}}{=}  \oP(1) \label{eq:aux5}
\end{align}
where the last estimate follows from $\delta_n \mu^n \left(Q_n\right) \searrow 0$. Furthermore, as $\mathcal V$ contains only scales $\leq \delta_n^{-2} \left|Q_n\right|$ we obtain that
\begin{align}
&\max\limits_{Q \in \mathcal V} \left[ \left| \left|Q\right|^{-\frac12} \sum\limits_{i \in Q}X_i \right| - \pen{v}{\left|Q\right|}\right] \nonumber\\
\leq &\max\limits_{Q \in \mathcal V} \left[ \left| \left|Q\right|^{-\frac12} \sum\limits_{i \in Q}X_i \right| - \pen{v}{\delta_n^{-2} \left|Q_n\right|}\right] \nonumber\\[0.1cm]
\stackrel{\mathrm{(a)}}{=}& \OP \left(\sqrt{\ln\left(\left|Q_n\right|\right)} + \sqrt{\ln\left(-\ln\left(m\left(\theta_1^n\right)-m\left(\theta_0\right)\right)\right)}\right) - \pen{v}{\delta_n^{-2} \left|Q_n\right|} \nonumber\\[0.1cm] 
= &\oP(1), \label{eq:aux6}
\end{align}
where we used {$\left|Q_n\right| = o\left(n^\beta\right)$} with {$\beta>0$} sufficiently small. Consequently
\begin{align*}
& \P{}{\max\limits_{Q \in \mathcal T'} \gamma\left(Q\right)\leq q} \\
\stackrel{\eqref{eq:aux5}}{\quad=\quad}& \P{}{\max\limits_{Q \in \mathcal T'} \left[ \left| \left|Q\right|^{-\frac12} \sum\limits_{i \in Q \setminus Q_n} v_i X_i \right| - \pen{v}{\left|Q\right|}\right] \leq q} + o(1)\\
\stackrel{\mathrm{(b)}}{\quad=\quad}& \P{}{\max\limits_{Q \in \mathcal T'} \left[ \left| \left|Q\right|^{-\frac12} \sum\limits_{i \in Q}X_i \right| - \pen{v}{\left|Q\right|}\right] \leq q}  + o(1)\\ 
\stackrel{\eqref{eq:aux6}}{\quad=\quad}& \P{}{\max\limits_{Q \in {\mathcal Q_{n| r_n}^{\mathrm{MS}} }} \left[ \left| \left|Q\right|^{-\frac12} \sum\limits_{i \in Q}X_i \right| - \pen{v}{\left|Q\right|}\right] \leq q}  + o(1)\\ 
\quad=\quad& \P{}{M_n \left({\mathcal Q_{n| r_n}^{\mathrm{MS}} }\right) \leq q}  + o(1)
\end{align*}
which yields the claim.
\end{proof}

\begin{proof}[Proof of Theorem \ref{thm:asymptPower}(b)]
For the multiscale procedure we have to compute a lower bound for $\P{\theta^n}{ T_n (Y, \theta_0, {\mathcal Q_{n| r_n}^{\mathrm{MS}} }) > q}$. Similar to the Proof of Theorem \ref{thm:asymptPower}(a) we obtain
\begin{align*}
&\P{\theta^n}{ T_n (Y, \theta_0, {\mathcal Q_{n| r_n}^{\mathrm{MS}} }) > q} \\
\geq& \P{}{\left\{\max_{Q \in \mathcal T'}\gamma \left(Q\right) \leq q\right\} \cap \left\{\gamma \left(Q_n\right) > q \right\}} + \P{}{\max_{Q\in\mathcal T'}\gamma \left(Q\right) > q} + o(1).
\end{align*}
By Lemma \ref{lem:distribution_unchanged_MS}(b) we furthermore get
\begin{align*}
&\P{}{\max_{Q \in \mathcal T'}\gamma \left(Q\right) \leq q} \\
=& \P{}{\max_{Q \in \mathcal T'}\left[\left| \mu^n \left(Q\right) + \frac{1}{\sqrt{\left|Q\right|}} \sum\limits_{i \in Q} v_i X_i \right| - \pen{v}{\left|Q\right|}\right] \leq q} \\
=& \P{}{\max_{Q \in \mathcal T'}\left[\left| \mu^n \left(Q\right) + \frac{1}{\sqrt{\left|Q\right|}} \sum\limits_{i \in Q \setminus Q_n} v_i X_i \right| - \pen{v}{\left|Q\right|}\right] \leq q} + o(1),
\end{align*}
which shows by independence that
\begin{align*}
&\P{}{\left\{\max_{Q \in \mathcal T'}\gamma \left(Q\right) \leq q\right\} \cap \left\{\gamma \left(Q_n\right) > q \right\}}\\
=& \P{}{\max_{Q \in \mathcal T'}\gamma \left(Q\right) \leq q} \P{}{\gamma \left(Q_n\right) > q } + o(1).
\end{align*}
Now the proof can be concluded as the one of Theorem \ref{thm:asymptPower}(a).
\end{proof}

\begin{proof}[Proof of Corollary \ref{cor:power}]
The procedures have asymptotic power $1$ if and only if 
\[
F\left(\bar q + \sqrt{-2{v}\log\left(a_n\right)},n^{d/2}\sqrt{a_n}\frac{m\left(\theta_1^n\right) - m\left(\theta_0\right)}{\sqrt{v\left(\theta_0\right)}},\frac{v\left(\theta_1^n\right)}{v\left(\theta_0\right)}\right) \to 1
\]
as $n \to \infty$ with $\bar q \in \left\{q^*,q\right\}$ respectively. The straight-forward estimate
\[
F\left(x,\mu,\sigma^2\right) \geq\max\left\{\Phi\left(\frac{-x-\mu}{\sigma}\right), \Phi\left(\frac{\mu-x}{\sigma}\right) \right\}
\]
shows that this is the case if and only if
\[
\frac{x+\mu}{\sigma} \to -\infty \qquad\text{or}\qquad\frac{x-\mu}{\sigma} \to -\infty.
\]
Inserting the values for $x$, $\mu$ and $\sigma$ and noting that $q^*,q$ are uniformly bounded by the $(1-\alpha)-$quantile of $M\left(\mathcal Q^*,v\right)$ gives the claim.
\end{proof}

\appendix

\section{Packing numbers of Example \ref{ex:sets}}\label{appA}
The computation of the packing numbers given in Example \ref{ex:sets} will be done by means of the covering number. The covering number $\mathcal N \left(\epsilon, \rho, \mathcal W\right)$ of a subset $\mathcal W \subset \mathcal R^*$ w.r.t. a metric $\rho$ is given by the minimal number of balls of radius $\epsilon>0$ needed to cover $\mathcal W$ \citep[cf.][Def. 2.2.3]{VaartWellner:1996}. It is immediately clear, that
\[
\mathcal N \left(\epsilon, \rho, \mathcal W\right) \leq \mathcal K \left(\epsilon, \rho, \mathcal W\right) \leq \mathcal N \left(\frac{\epsilon}{2}, \rho, \mathcal W\right),
\]
and hence it sufficed to compute $\mathcal N \left(\left(\delta u\right)^{1/2}, \rho^*, \left\{R \in \mathcal R^* ~\big|~\left| R \right| \leq \delta\right\}\right)$ with $\rho^*$ as in \eqref{eq:metric} to show \eqref{eq:capacity_bound}. In the following, we will use the notation from Example \ref{ex:sets}.

\begin{lem}\label{lem:packing_hyperrectangles}
For any $\epsilon>0$ there exists a constant $C$ depending only on the dimension $d$ and $\epsilon$ such that for all $u, \delta \in \left(0,1\right]$ it holds
\begin{align*}
\mathcal{K}\left((\delta u)^{1/2}, \rho^*, \{S \in \mathcal{S}^*: |S|\leq \delta\}\right)
&\leq {C_1 \delta^{-(2d-1)} \left( \log (1/\delta)\right)^{d-1} u^{-2d} \left( \log (1/u)\right)^{d-1}}\\
&\leq  C u^{-(2d+\epsilon)}\delta^{-(2d-1+\epsilon)},
\end{align*}
i.e. \eqref{eq:capacity_bound} holds true with $k_1 = C, k_2 = 2d+\epsilon$ and $V_{\mathcal S^*} = 2d-1+\epsilon$.
\end{lem}

\begin{remark}
	Note that it can even be shown that there are $k_1, k_2 > 0$ such that \eqref{eq:capac_bound} with additional powers of $\left(-\log\left(\delta\right)\right)$ on the right-hand side is satisfied with $v = 1$, see e.g. Theorem 1 in \citet{Walther} or Lemma 2.1 in \cite{ds18}.
\end{remark}

\begin{proof}[Proof of Lemma \ref{lem:packing_hyperrectangles}]
We approximate the hyper-rectangles in $\mathcal W = \{S \in \mathcal{S}^*: |S|\leq \delta\}$ by hyper-rectangles with vertices in the lattice $\mathbb{L}_m:= \left\lbrace \frac{i}{m} ~\big|~ i=0, \ldots, m\right\rbrace^d$ where $m$ has to be specified later. The set of all hyper-rectangles with such vertices and size $\leq \delta$ will be denoted by $\mathcal W'_m$. For $S \in \mathcal W$ denote by $k_1, \ldots, k_d$ the edge lengths. Then $\prod_{j=1}^d k_j  \leq \delta$ and $k_i \leq 1$. It is immediately clear that there exists an approximating hyper-rectangle $S' \in \mathcal W'_m$ such that
\begin{align}
\left(\rho^*\left(S,S'\right)\right)^2 &= \left|S \bigtriangleup S'\right|\nonumber\\
&\leq 2 \left( k_2 \cdot ... \cdot k_d +k_1 \cdot k_3 \cdot ...\cdot k_d +  \ldots + k_1 \cdot \ldots \cdot k_{d-1}\right) \cdot \frac{1}{2m}\nonumber\\
&\leq \frac{d}{m}.\label{eq:covering_aux1}
\end{align}
Hence, we obtain $\rho^*\left(S,S'\right) \leq \left(\delta u \right)^{1/2}$ if we choose $m: =  \frac{d}{\delta u}$. Now we have to compute the cardinality of $\mathcal W'_{d/(\delta u)}$. First note that the number of possible left bottom vertices is bounded from above by $m^d = \# \mathbb L_m$. If we denote the edge lengths of $S' \in \mathcal W'_m$ by $l_1, \ldots, l_d$, we can find integers $i_1, ..., i_d$ such that $l_j = \frac{i_j}{m}$ and $\prod_{j=1}^d i_j \leq \delta m^d =:N$. Therefore, we obtain
\begin{equation}\label{eq:covering_number_S}
\mathcal N \left((\delta u)^{1/2}, \rho^*, \{S \in \mathcal{S}^*: |S|\leq \delta\}\right) \leq \# \mathcal W'_{m} \leq m^d \cdot \# \mathcal{C}_N
\end{equation}
with $\mathcal C_N := \left\{(i_1, \ldots, i_d) \in \N^d ~\big|~\prod_{j=1}^d i_j \leq N\right\}$. 
To compute $\# \mathcal C_N$, we employ Minkowski's theorem \citep[cf.][Sec. III.2.2]{c97}, which ensures that the Lebesgue volume $\Delta_d(N)$ of \linebreak
$\left\{(x_1, \ldots, x_d) \in [1, N]^d ~\big|~x_1 \cdot \ldots \cdot x_d \leq N \right\}$ is comparable with $\#\mathcal C_N$ up to a factor of $2^d$. 
{We show by induction that
\begin{align*}
\Delta_d(N)  = \frac{1}{(d-1)!} N \left(\log N \right)^{d-1}.
\end{align*} 
\begin{proof}
 \textsf{$d=1$:}\quad
  $\Delta_1= N.$\\
 \textsf{$d \mapsto d+1:$}\quad
  $x_1 \cdot \ldots \cdot x_{d+1} \leq N \Leftrightarrow x_1 \cdot \ldots \cdot x_d \leq \frac{N}{x_{d+1}},$ where $x_{d+1}\in [1, N].$ Hence,
 \begin{align*}
 \Delta_{d+1}(N) &= \int_1^N \Delta_d \left(\frac{N}{x_{d+1}} \right) dx_{d+1}\\
 &= \frac{1}{(d-1)!} \int_1^N \frac{N}{x_{d+1}} \left(\log \left( \frac{N}{x_{d+1}}\right) \right)^{d-1} dx_{d+1}\\
 &= \frac{N}{(d-1)!} \left(- \int_N^1 y(\log y)^{d-1} \frac{1}{y^2} dy\right)\\
 &= \frac{N}{(d-1)!} \int_1^N \left(\log y \right)^{d-1} \frac{1}{y} dy\\
 &= \frac{1}{d!} N (\log N)^d. 
 \end{align*}  
\end{proof}
}
Inserting this into \eqref{eq:covering_number_S}, we obtain
\begin{align*}
&\mathcal N \left((\delta u)^{1/2}, \rho^*, \{S \in \mathcal{S}^*: |S|\leq \delta\}\right) \\
\leq& m^d * \# \mathcal{P}_N \\
\leq& 2^d m^d \Delta_d\left(\delta m^d\right) \\
\lesssim& \delta m^{2d} {\left[ \log \left(\delta m^d\right)\right]}^{d-1}\\
=&\delta^{-(2d-1)} u^{-2d} \left[\log \left( d^d \delta^{-(d-1)} u^{-d}\right) \right]^{d-1}\\
\lesssim&\delta^{-(2d-1)} \left( \log (1/\delta)\right)^{d-1} u^{-2d} \left( \log (1/u)\right)^{d-1},
\end{align*}
where we used $(x+y)^{d-1} \leq c x^{d-1} y^{d-1}$ for $x,y\geq 1$. This proves the claim.
\end{proof} 
 
\begin{lem}\label{lem:packing_hypercubes}
There exists a constant $C$ depending only on the dimension $d$ such that for all $u, \delta \in \left(0,1\right]$ it holds
\[
\mathcal{K}\left((\delta u)^{1/2}, \rho^*, \{Q \in \mathcal{Q}^*: |Q|\leq \delta\}\right)\leq C \delta^{-1}u^{-(d+1)},
\]
i.e. \eqref{eq:capacity_bound} holds true with $k_1 = C, k_2 = d+1$ and $V_{\mathcal Q^*} = 1$.
\end{lem}

\begin{proof}
We proceed as in the Proof of Lemma \ref{lem:packing_hyperrectangles}. In contrast to hyper-rectangles, we obtain here instead of \eqref{eq:covering_aux1} the better estimate 
\begin{align*}
\left(\rho^*\left(Q,Q'\right)\right)^2 = \left|Q \bigtriangleup Q'\right| \leq \frac{d \delta^{\frac{d-1}{d}}}{m}.
\end{align*}
as all edges have the same length, i.e. we can choose $m: =  \frac{d}{\delta^{1/d} u}$. Furthermore, the cardinality of $\mathcal W'_m$ is bounded by the number of lower left vertices times the number of possibilities for an adjacent vertex, which gives
\[
\# \mathcal W'_m \leq m^{d} \cdot \left(\delta^{1/d}m\right)=m^{d+1}\delta^{1/d}.
\]
Therefore we finally obtain
\begin{align*}
\mathcal N \left((\delta u)^{1/2}, \rho^*, \{S \in \mathcal{S}^*: |S|\leq \delta\}\right) &\leq \# \mathcal W'_{d/(\delta^{1/d} u)} \\
& \leq \left(\frac{d}{\delta^{1/d} u}\right)^{d+1}\delta^{1/d} = d^{d+1} u^{-(d+1)} \delta^{-1},
\end{align*}
which proves the claim.
\end{proof}

\begin{lem}\label{lem:packing_halfspaces}
There exists a constant $C$ depending only on the dimension $d$ such that for all $u, \delta \in \left(0,1\right]$ it holds
\[
\mathcal{K}\left((\delta u)^{1/2}, \rho^*, \{H \in \mathcal{H}^*: |H|\leq \delta\}\right)\leq C \delta^{-2}u^{-2},
\]
i.e. \eqref{eq:capacity_bound} holds true with $k_1 = C, k_2 = 2$ and $V_{\mathcal H^*} = 2$.
\end{lem}
\begin{proof}
{For two points $a_1, a_2 \in\mathbb S^{d-1}$ we denote by $\sphericalangle \left(a_1,a_2\right) := \arccos \left(\left\langle a_1, a_2\right\rangle\right) \in \left[0,\pi\right]$ the spherical angle between $a_1, a_2$. Now l}et $\mathcal W_{N,m} = \left\{H_{a_i,\alpha_j} ~\big|~ i=1,...,N, j=1,...,m\right\}$ with numbers $a_1,...,a_N \in \mathbb S^{d-1}$ and $\alpha_1,...,\alpha_m \in \left[0,\sqrt{d}\right]$. Note that $H_{a,\alpha} = \emptyset$ for $\alpha > \sqrt{d}$ by definition and Pythagoras' theorem. It is convenient to choose $\alpha_1, ..., \alpha_m$ as equidistant, e.g.
\[
\alpha_i := \frac{i-\frac{1}{2}}{m} \sqrt{d}, \quad i=1, \ldots, m.
\]
Furthermore we choose $a_1,...,a_N$ as a maximal system of points in $\mathbb S^{d-1}$ such that $\sphericalangle (a_j, a_k) \geq \left( \frac{1}{m}\right)^{\frac{1}{d-1}}$ for all $j \neq k$. This implies that
\[
\mathbb{S}^{d-1} \subset \bigcup\limits_{j=1}^N S_{a_j}\left( \left(\frac{1}{m}\right)^{\frac{1}{d-1}} \right)
\]
with the spherical cap $S_{a} \left(\theta_0\right) = \left\{e \in \mathbb S^{d-1} ~\big|~ \sphericalangle (a, e) \leq \theta_0\right\}$. Note that
\[
\left|S_a \left(\theta_0\right)\right| \sim \frac{\int\limits_0^{\theta_0} \left( \sin t \right)^{d-2} \,\mathrm dt}{\int\limits_0^\pi \left( \sin t\right)^{d-2} \,\mathrm dt} \sim \theta_0^{d-1} 
\]
for small values of $\theta_0$. Now, for any given $a \in \mathbb{S}^{d-1}$ and $\alpha \in \left[0, \sqrt{d}\right]$, we can find $1 \leq i \leq N$ and $1 \leq j \leq m$ such that
\begin{align*}
{\sphericalangle (a, a_i)} \leq \left( \frac{1}{m}\right)^{\frac{1}{d-1}}, \qquad \left|\alpha-\alpha_j \right| \leq \frac{\sqrt{d}}{m}.
\end{align*}
Now we split
\begin{align*}
\left(\rho^* \left(H_{a,\alpha}, H_{a_i,\alpha_j}\right)\right)^2 
& \leq \left|H_{a, \alpha}\bigtriangleup H_{a_i, \alpha} \right| + \left|H_{a_i, \alpha}\bigtriangleup H_{a_i, \alpha_j} \right|
\end{align*}
and since $H_{a_i, \alpha}\bigtriangleup H_{a_i, \alpha_j} $ is a $d-1$-dimensional space of width $\leq \frac{\sqrt{d}}{m}$ and $H_{a, \alpha}\bigtriangleup H_{a_i, \alpha}$ is a union of hyperpyramids with opening angle $\leq \left( \frac{1}{m}\right)^{\frac{1}{d-1}}$, we obtain
\[
\left(\rho^* \left(H_{a,\alpha}, H_{a_i,\alpha_j}\right)\right)^2 \leq \frac{C}{m}
\]
where $C$ is some generic constant depending only on $d$. Hence if we choose $m = C^{-1}\delta^{-1} u^{-1}$, then for each $H \in \{H \in \mathcal{H}^*: |H|\leq \delta\}$ there exists $H' \in \mathcal W_{N,m}$ such that $\rho^*\left(S,S'\right) \leq \left(\delta u \right)^{1/2}$. Now we have to estimate $N$. By elementary geometry it follows that
\[
\bigcup\limits_{j=1}^N S_{a_j}\left(\frac{1}{2}\left(\frac{1}{m}\right)^{\frac{1}{d-1}} \right) \subset \mathbb{S}^{d-1} \subset \bigcup\limits_{j=1}^N S_{a_j}\left( \left(\frac{1}{m}\right)^{\frac{1}{d-1}} \right),
\]
and furthermore up to boundary points, the sets on the left-hand side are disjoint. Therefore we obtain for the volumes that
\[
N \left|S_{a_j}\left(\frac{1}{2}\left(\frac{1}{m}\right)^{\frac{1}{d-1}} \right) \right|\leq \left|\mathbb{S}^{d-1}\right| \leq N \left|S_{a_j}\left( \left(\frac{1}{m}\right)^{\frac{1}{d-1}} \right)\right|
\]
which implies $N \sim m$. Consequently, $\# \mathcal W_{N,m} \sim m^2$ which proves the claim.

\end{proof}

\section*{Acknowledgements}

Financial support by the German Research Foundation DFG of CRC 755 A04 is acknowledged. We thank Katharina Proksch for helpful comments on the proof of Theorem \ref{thm:approximation} as well as Guenther Walther and three anonymous referees for several constructive comments which helped us to improve the presentation of the paper substantially.

\bibliography{mybib}
\bibliographystyle{apalike}

\end{document}